\documentclass[12pt,reqno]{amsart}

\usepackage[colorlinks=true,urlcolor=blue,bookmarks=true,bookmarksopen=true,citecolor=blue]{hyperref}

\usepackage{mathtools} 
\usepackage{graphicx}
\usepackage[usenames,dvipsnames]{color}
\usepackage[all]{xy}
\usepackage{amsmath}
\usepackage{amsfonts,amssymb}
\usepackage{amsthm}
\usepackage{bbm}
\usepackage{mathrsfs}
\usepackage{marginnote}
\usepackage{pdfsync}
\usepackage{xfrac}
\usepackage{varioref}
\usepackage{a4wide}
\usepackage{url}

\makeatletter
\newcommand{\rmnum}[1]{\romannumeral #1}
\newcommand{\Rmnum}[1]{\expandafter\@slowromancap\romannumeral #1@}
\makeatother

\usepackage{parskip}

\usepackage{bbm}

\newcommand{\id}{id}

\makeatletter
\DeclareFontFamily{OMX}{MnSymbolE}{}
\DeclareSymbolFont{MnLargeSymbols}{OMX}{MnSymbolE}{m}{n}
\SetSymbolFont{MnLargeSymbols}{bold}{OMX}{MnSymbolE}{b}{n}
\DeclareFontShape{OMX}{MnSymbolE}{m}{n}{
    <-6>  MnSymbolE5
   <6-7>  MnSymbolE6
   <7-8>  MnSymbolE7
   <8-9>  MnSymbolE8
   <9-10> MnSymbolE9
  <10-12> MnSymbolE10
  <12->   MnSymbolE12
}{}
\DeclareFontShape{OMX}{MnSymbolE}{b}{n}{
    <-6>  MnSymbolE-Bold5
   <6-7>  MnSymbolE-Bold6
   <7-8>  MnSymbolE-Bold7
   <8-9>  MnSymbolE-Bold8
   <9-10> MnSymbolE-Bold9
  <10-12> MnSymbolE-Bold10
  <12->   MnSymbolE-Bold12
}{}

\let\llangle\@undefined
\let\rrangle\@undefined
\DeclareMathDelimiter{\llangle}{\mathopen}%
                     {MnLargeSymbols}{'164}{MnLargeSymbols}{'164}
\DeclareMathDelimiter{\rrangle}{\mathclose}%
                     {MnLargeSymbols}{'171}{MnLargeSymbols}{'171}
\makeatother
\renewenvironment{proof}[1][\proofname]{\noindent{\bfseries #1.  }}{\qed}

\newcommand{\Z}{\mathbbm{Z}}                     
\newcommand{\R}{\mathbbm{R}}                     
\newcommand{\T}{\mathbbm{T}}                     

\newcommand{\spec}{\mathrm{Spec}}               
\newcommand{\crit}{\mathrm{Crit\,}}               
\newcommand{\hess}{\mathrm{Hess\,}}             
\newcommand{\Cont}{\mathrm{Cont}}     
\newcommand{\RFH}{\mathrm{RFH}}    
\newcommand{\HF}{\mathrm{HF}}    
\newcommand{\CF}{\mathrm{CF}}

\newcommand{\h}{\mathrm{H}}

\newtheorem{thm}{Theorem}[section]               
\newtheorem*{thm*}{Theorem}               
\newtheorem{cor}[thm]{Corollary}        
\newtheorem*{cor*}{Corollary}        
\newtheorem{lem}[thm]{Lemma}            
\newtheorem{prop}[thm]{Proposition}     
\theoremstyle{definition}
\newtheorem{defn}[thm]{Definition}      
\newtheorem{rem}[thm]{Remark}           
\newtheorem{ex}[thm]{Example}           
 \newtheorem*{acknowledgement*}{\protect\acknowledgementname}
\newenvironment{claim}[1]{\par\noindent\underline{Claim:}\space#1}{}

 \providecommand{\acknowledgementname}{Acknowledgement}

\title{Translated points on hypertight contact manifolds}

\author{Matthias Meiwes and Kathrin Naef} 
\date{} 

\address{Matthias Meiwes\\
Department of Mathematics\\
WWU M\"unster}
\email{\texttt{matthias.meiwes@uni-muenster.de}}

\address{Kathrin Naef\\
Department of Mathematics\\
ETH Z\"urich}
\email{\texttt{kathrin.naef@math.ethz.ch}}

\date{\today}

\begin{document}

\begin{abstract}
	A contact manifold admittting a supporting contact form without contractible Reeb orbits is called hypertight. In this paper we construct a Rabinowitz Floer homology associated to an arbitrary supporting contact form for a hypertight contact manifold $\Sigma$, and use this to prove versions of conjectures of Sandon \cite{Sandon2013} and Mazzucchelli \cite{Mazzucchelli2015} on the existence of translated points and invariant Reeb orbits, and to show that positive loops of contactomorphisms give rise to non-contractible Reeb orbits.

\end{abstract}
\maketitle

\section{Introduction}
Floer theory has been used to define invariants for contact manifolds. One type of Floer homology is the so called Rabinowitz Floer homology, which was introduced by Cieliebak and Frauenfelder \cite{CieliebakFrauenfelder2009} and whose chain complex is generated by closed Reeb orbits. 

In order to define a Floer homology one always needs to show compactness for certain moduli spaces. Let $\left(\Sigma,\xi\right)$ be a closed coorientable contact manifold and $\alpha$ a supporting contact form. We want to define Rabinowitz Floer homology on the symplectisation $\Sigma\times(0,+\infty)$. Unfortunately compactness typically fails at the concave end.

A contact manifold $\left(\Sigma,\xi\right)$ is called \textbf{hypertight} if it admits a supporting contact form, say $\alpha_0$, without any contractible  Reeb orbits.  For instance, the three torus $\T^3$ equipped with the contact structure $\alpha_k=\cos(2\pi kr)\,ds+\sin(2\pi k r)\,dt$ is a hypertight contact manifold. Certain prequantisation spaces are hypertight as well, see Example \ref{ex:prequant} below. More constructions of hypertight contact manifolds can be found in \cite{ColinHonda2005}. Using an SFT-compactness result, Albers, Fuchs and Merry \cite{AlbersFuchsMerry2013} constructed Rabinowitz Floer homology on the symplectisation provided that the contact form is without contractible Reeb orbits. The main point of their argument is, that if a flow line would escape to $-\infty$ it converges to contractible Reeb orbits.


In this paper, we define Rabinowitz Floer homology $\RFH_*( \Sigma, \alpha)$ for \textit{any} contact form $\alpha$ supporting  a hypertight contact structure,  including those which do have contractible Reeb orbits. 
This Rabinowitz Floer homology $\RFH_*(\Sigma, \alpha)$ does not depend on the choice of supporting contact form. More precisely, denoting  by $\mathcal{C}(\xi)$ the set of all supporting contact forms of $\xi$, we construct an isomorphism between $\RFH_*\left(\Sigma,\alpha_1\right)$ and $\RFH_*\left(\Sigma,\alpha_2\right)$ for any two $\alpha_1,\alpha_2\in\mathcal{C}(\xi)$. 

We denote by $\Cont_0\left(\Sigma,\xi\right)$ the space of all contactomorphisms on $\Sigma$ which are contact-isotopic to the identity and similarly $\mathcal{P}\Cont_0\left(\Sigma,\xi\right)$ the space of all paths $\hat{\varphi}=\left\{\varphi_t\right\}_{t\in[0,1]}$ of contactomorphisms which start at the identity.

Let $\varphi\in\Cont_0\left(\Sigma,\xi\right)$ and $\alpha\in\mathcal{C}(\xi)$. Sandon \cite{Sandon2013} introduced the important notion of  translated points: a point $x\in\Sigma$ is a \textbf{translated point of $\varphi$ with respect to $\alpha$}, if there exists $\eta\in\R$ such that
\[
\varphi(x)=\varphi^\eta_R(x) \quad\text{and}\quad \left.\varphi^*\alpha\right|_x=\left.\alpha\right|_x.
\]
We stress that the definition of a translated point depends on the contact form.
	
 We denote by $\RFH_*\left(\Sigma,\alpha_1;\hat{\varphi}\right)$ the Rabinowitz Floer homology of the pair $(\hat{\varphi};\alpha)$ whose chain complex has as generators orbits where first one flows along a Reeb orbit and then along the path $\hat{\varphi}$ until one hits the Reeb orbit again. In particular, if $\hat{\varphi}=\id$ the generators are closed Reeb orbits, and (by definition) $\RFH_*( \Sigma, \alpha ; \id) = \RFH_*( \Sigma ,\alpha)$.

The main technical result of this paper is the following.
\begin{thm}\label{mainthm}
 Let $\left(\Sigma,\xi\right)$ be a hypertight contact manifold and $\hat{\varphi}\in\mathcal{P}\Cont_0\left(\Sigma,\xi\right)$. 
 For any two $\alpha_1,\,\alpha_2\in\mathcal{C}(\xi)$, the Rabinowitz Floer homology groups $\RFH_*( \Sigma, \alpha_i ; \hat{\varphi})$ are well defined, and there are isomorphisms 
 	\[
 	\RFH_*\left(\Sigma,\alpha_1;\hat{\varphi}\right)\cong\RFH_*\left(\Sigma,\alpha_2;\hat{\varphi}\right)\cong\mathrm{H}_{*+n-1}\left(\Sigma;\Z_2\right).
 	\]
 	
\end{thm}


The main application of Theorem \ref{mainthm} resolves certain cases of conjectures of Sandon and Mazzucchelli on the existence of translated points.

\begin{rem}
Translated points are a special case of leafwise intersection points. The study of leafwise intersection points was initiated by Moser \cite{Moser1978} and since then many authors have proved existence results, e.g. \cite{Banyaga1978}, \cite{EkelandHofer1989}, \cite{Hofer1990}, \cite{Ginzburg2007}, \cite{Ziltener2010}, \cite{AlbersFrauenfelder2010}, \cite{Sandon2013}. Albers and Frauenfelder \cite{AlbersFrauenfelder2010c} set up the variational problem so that it targets precisely leafwise intersections, respectively translated points.
\end{rem}

The following proves versions of a conjecture by Sandon \cite[Conjecture 1.2]{Sandon2013} for certain contact manifolds.


\begin{thm}\label{translated points for hypertight}
	Let $(\Sigma,\xi)$ be a hypertight contact manifold. Then:
	\begin{enumerate}
\item[\rmnum{1})] For any $\alpha\in\mathcal{C}(\xi)$ and for any $\varphi\in\Cont_0\left(\Sigma,\xi\right)$ there exists a translated point of $ \varphi$ with respect to $\alpha$.
\item[\rmnum{2})] For a generic pair $(\alpha, \varphi)$, $\alpha$ and $\varphi$ as above, the number of translated points of $\varphi$ with respect to $\alpha$ is bounded from below by 
$\sum_{i=0}^{\dim(\Sigma)}\dim \mathrm{H}_i(\Sigma ;\Z_2)$.
\end{enumerate}
\end{thm}


The above result is already known for supporting contact forms $\alpha$ without contractible Reeb orbits, \cite{AlbersFuchsMerry2013}. Theorem \ref{translated points for hypertight} extends this to \emph{all} supporting contact forms. In \cite{Sandon2013} similar results were proved for the specific contact manifolds $S^{2n-1}$ and $\R \mathrm{P}^{2n-1}$, equipped with their standard contact forms.

\begin{rem}\phantomsection \label{rem:translated points}$ $
	\begin{enumerate} 
		\item[\rmnum{1})] The second statement of Theorem \ref{translated points for hypertight} can be improved: for \emph{any} $\alpha$ and generic $\varphi$ one of the  following holds: (a) there is a translated point on a closed contractible Reeb orbit or (b) there are at least  $\sum_{i=0}^{\dim(\Sigma)}\dim \mathrm{H}_i(\Sigma ;\Z_2)$ many translated points.


\item[\rmnum{2})] If the oscillation norm of the associated contact Hamiltonian (cf. Definition \ref{def:contactham}) of $\varphi$ is smaller than the smallest contractible Reeb period of the contact form option (b) above is always true. 		
\end{enumerate}				
\end{rem}


\begin{ex}
\label{ex:prequant}
 An important class of examples of hypertight manifolds come from certain \textbf{prequantisation spaces}. Let $\left(M,\omega\right)$ be a closed symplectic manifold and assume that the de Rham cohomology class $[\omega]$ has a primitve integral lift in $H^2\left(M;\Z_2\right)$. Now, we look at the circle bundle $p:~\Sigma_k\rightarrow M$ with corresponding Euler class $k[\omega],~0\neq k\in\Z$ and connection $1$-form $\alpha$ with $p^*(k\omega)=-d\alpha$. Then, $\left(\Sigma_k,\alpha\right)$ is a contact manifold with periodic Reeb flow. The closed Reeb orbits are the fibres of the bundle. Moreover, the long exact sequence of the fibration
 \[
  \pi_2\left(M\right)\overset{q_k}{\rightarrow}\pi_1\left(S^1\right)\rightarrow\pi_1\left(\Sigma_k\right)\rightarrow\pi_1\left(M\right)\rightarrow 0.
 \]
 shows that the map $q_k$ is non-trivial if and only if the homotopy class of the fibre is torsion. Note that if $q_k$ is non-trivial, then $q_{nk}$ is non-trivial for each $n\neq 0$.
 It follows from \cite[Theorem 1.5]{AlbersFuchsMerry2013} that a prequantisation space is hypertight if the fibre is not torsion.

\end{ex}


Another application of Theorem \ref{mainthm} concerns the  existence of $\varphi$-invariant Reeb orbits. We define $\mathcal{S}\Cont_0(\Sigma,\alpha):=\left\{\left.\varphi\in\Cont_0(\Sigma,\xi)\right|\varphi^*\alpha=\alpha\right\}$ to be the set of strict contactomorphisms in $\Cont_0(\Sigma,\xi)$ with respect to the supporting contact form $\alpha$. For $\varphi\in \mathcal{S}\Cont_0\left(\Sigma,\xi\right)$, a Reeb orbit $x:~\R\rightarrow\Sigma$ is called \textbf{$\varphi$-invariant} if $\varphi(x(t))=x(t+\tau)$ for some $\tau\in\R\setminus\left\{0\right\}$.  In \cite[Conjecture 1.2]{Mazzucchelli2015} Mazzucchelli conjectures that for $\varphi\in \mathcal{S}\Cont_0\left(\Sigma,\xi\right)$ there always is a $\varphi$-invariant Reeb orbit.
Specialising Theorem \ref{translated points for hypertight} i) to strict contactomorphims, we prove Mazzucchelli's conjecture for hypertight contact manifolds:

\begin{cor}
	Let $\left(\Sigma,\xi\right)$ be a hypertight contact manifold. Let $\alpha\in\mathcal{C}(\xi)$ and fix $\varphi\in \mathcal{S}\Cont_0\left(\Sigma,\alpha\right)$. Then either there exists a $\varphi$-invariant Reeb orbit or an entire Reeb orbit is left fixed by $\varphi$.
\end{cor}



Another application of Theorem \ref{mainthm} is the study of the existence of \emph{non-contractible} closed Reeb orbits. 
Given a loop $\hat{\varphi}=\left\{\varphi_t\right\}_{t\in[0,1]}\in\mathcal{P}\Cont_0\left(\Sigma,\xi\right)$, let us denote by $u_{\hat{\varphi}}\in [S^1, \Sigma]$ the free homotopy class of the loop $t\mapsto\varphi_t(x)$. Recall a  loop of contactomorphisms is called \textbf{positive} if the associated contact Hamiltonian is positive, see Definition \ref{def:contactham}. In \cite{AlbersFuchsMerry2013} it was shown that on hypertight contact manifolds, there do not exist any \emph{contractible} positive loops of contactomorphisms.

\begin{thm}\label{thm:noncontractible orbits}
  Let $(\Sigma,\xi)$ be a hypertight contact manifold. Assume there exists a positive loop $\hat{\varphi}$ of contactomorphisms. Then the class $u_{\hat{\varphi}}$ is a non-trivial element of $[S^1, \Sigma]$, and for any $\alpha\in \mathcal{C}(\xi)$ there exists a closed Reeb orbit of $\alpha$ in the free homotopy class of $-u_{\hat{\varphi}}$ (which is thus necessarily non-contractible).
  
  
\end{thm}

Note that again this result holds for any supporting contact form.

\begin{ex}
	If a contact manifold $\left(\Sigma,\xi\right)$ admits a supporting contact form with periodic Reeb flow then the Reeb flow itself constitutes a positive loop. Thus prequantisation spaces always admit a positive loop. If in addition the fibre is not torsion then they are examples of hypertight contact manifolds with a positive loop of contactomorphisms.
\end{ex}

\begin{rem}
In a sequel to the present paper, we extend the constructions in this paper to deduce analogous results for dynamically convex manifolds. Here a contact manifold is called \textbf{dynamically convex} if the normalised Conley Zehnder index of every closed contractible Reeb orbit is nonnegative, see \cite{HoferWysockiZehnder}.
\end{rem}

\acknowledgementname: We thank Peter Albers and Will Merry for many enriching discussions and helpful comments. We also thank Paul Biran for his advice.

\section{The perturbed Rabinowitz action functional}

In this paper, we always assume that the contact manifold is hypertight.

Let $\left(\Sigma,\xi\right)$ be a closed coorientable contact manifold with $\xi$ a hypertight contact structure. Let $\alpha_0\in\mathcal{C}(\xi)$ be a supporting contact form without contractible Reeb orbits. Denote by $R_0$ the Reeb vector field of $\alpha_0$ and by $\varphi_{R_0}^s:\Sigma\rightarrow\Sigma$ its Reeb flow. Let $\alpha_1\in\mathcal{C}(\xi)$ be any other supporting contact form, which possibly has contractible Reeb orbits. Then, there is a function $g:\Sigma\rightarrow (0,+\infty)$ such that $\alpha_1=g\cdot\alpha_0$.


Let $M:=\Sigma\times(0,+\infty)$. We want to equip $M$ with a symplectic form with primitive $\lambda$ such that $\lambda$ equals $r\alpha_1$ near $\Sigma\times\left\{1\right\}$ and $\lambda$ equals $r\alpha_0$ near $\Sigma\times \left\{0\right\}$. For this contact form to be symplectic it is crucial that we can homotope $r\alpha_1$ to $r\alpha_0$ in an increasing way along $r$. Let $0<\epsilon<\inf_{x\in\Sigma}g(x)$ and let $\nu>0$. Define
\[
 \lambda:=f(x,r)\alpha_0,
\]
where $f:M\rightarrow (0,+\infty)$ is defined as
\begin{equation}\label{lambda}
f(x,r)=
 \begin{cases}
  r g(x),&\text{ for } r> e^{-3\nu}\\
  r\epsilon,&\text{ for } r<e^{-4\nu}
 \end{cases}
\end{equation}
and such that $\frac{\partial f}{\partial r}>0$. 
A direct computation shows that $d\lambda$ is nondegenerate iff $\frac{\partial f}{\partial r}>0$.

We set 
\[
\Omega_{\nu}(\alpha_0, \alpha_1) = \left\{\lambda \in \Omega^1(M) \, | \, \lambda=f\alpha_0, \, f \, \text{satisfies }\eqref{lambda} \, \text{for some } \, \epsilon >0  \quad \text{and} \quad \frac{\partial f}{\partial r} > 0  \, \right\}.
\]
In the construction below we will fix $\lambda \in \Omega_{\nu}(\alpha_0, \alpha_1)$ after choosing a suitable $\nu>0$.

Suppose $\varphi:\Sigma\rightarrow\Sigma$ is a contactomorphism. Then there is a smooth positive function $\rho:\Sigma\rightarrow(0,+\infty)$ such that $\varphi^*\alpha_1=\rho\alpha_1$.

In the following, we always consider a contactomorphism which is contact-isotopic to the identity.

Let $\varphi\in\Cont_0\left(\Sigma,\xi\right)$, then there is a path $\hat{\varphi}=\left\{\varphi_t\right\}_{t\in[0,1]}$ with $\varphi_t=\mathbbm{1}$, for $t\in\left[0,\frac12\right]$ and $\varphi_1=\varphi$. We call such a path \textbf{admissible}. Also, there exists a smooth family of positive functions $\rho_t:\Sigma\rightarrow (0,+\infty)$ such that $\varphi_t^*\alpha_1=\rho_t\alpha_1$. 

\begin{defn}\label{def:contactham}
The \textbf{contact Hamiltonian} of $\hat{\varphi}$ with respect to $\alpha_1$ is the function $l:\Sigma\times[0,1]\rightarrow\R$ defined by
\begin{equation}\label{contactham}
 l_t\circ\varphi_t=\alpha_1\left(\frac{d}{dt}\varphi_t\right).
\end{equation}
\end{defn}

We extend $l$ to a Hamiltonian function $L:M\times [0,1] \rightarrow \R$ by 
\[
 L_t(x,r):=rl_t(x).
\]
The Hamiltonian diffeomorphism $\phi_L^t:M\rightarrow M$ associated to $L$ and the symplectic form $d\lambda$ is given by
\begin{equation}\label{hamdiffeo}
 \phi_L^t(x,r)=\left(\varphi_t(x),r\rho_t(x)^{-1}\right)
\end{equation}
on the interior of $\left\{\left.(x,r)\in M\right|~\lambda=r\alpha_1\right\}$ and thus in particular on $\Sigma\times(e^{-\nu},e^{\nu})$.

Moreover, let 
\[
 H(r)=\begin{cases}
       c,&\text{ for } r\in(e^{2\nu},+\infty)\\
       r-1,&\text{ for } r\in (e^{-\nu},e^\nu)\\
       -c,&\text{ for } r\in (0,e^{-2\nu})
      \end{cases}
\]
for some constant $c\geq\max\left\{1-e^{-\nu},e^\nu -1\right\}$ such that $H'(r)\geq 0$. We will also use $H$ for the function $H(x,r):=H(r)$ on $M$. Note that
\[
 X_H(x,r)=\frac{\partial H}{\partial r}(x,r)R_1(x)
\]
since $\frac{\partial H}{\partial r}=0$ on the region where $\lambda\neq r\alpha_1$.
Here, $R_1$ denotes the Reeb vector field of $\alpha_1$.

Now, let $\kappa:S^1\rightarrow\R$ be a smooth function with
\[
 \kappa(t)=0,\quad \text{for all}\quad t\in\left[\frac12,1\right]\quad\text{and}\quad \int_0^1\kappa(t)\,dt=1.
\]

We use $\kappa$ to modify the Hamiltonian $H$ to $\kappa(t)H(x,r)$.

In the following, we want to cutoff the Hamiltonian $L$.
 In order to do so, we have to take care that we cut off outside of the region where our perturbed functional will have its periodic orbits. Define a smooth function $\beta_\nu\in C^{\infty}\left(0,\infty),[0,1]\right)$ such that
\[
 \beta_\nu(r)=\begin{cases}
               1,&\quad r\in\left(e^{-\nu},e^{\nu}\right)\\
               0,&\quad r\in\left(0,e^{-2\nu}\right]\cup\left[e^{2\nu},+\infty\right)
              \end{cases}.
\]
We use $\beta_\nu$ to cutoff the Hamiltonian $L$ via $\beta_\nu L$.

Let $\mathcal{L}M$ denote the set of contractible smooth loops $u=(x,r):S^1\rightarrow M$.
Finally, we are ready to define the \textit{perturbed} Rabinowitz action functional associated to $\hat{\varphi}$,
\[\mathcal{A}^{(\hat{\varphi},\nu)}_{(\alpha_0,\alpha_1)}:\mathcal{L}M\times\R\rightarrow\R\]

\begin{equation}\label{functional}
 \mathcal{A}^{(\hat{\varphi},\nu)}_{(\alpha_0,\alpha_1)}(x,r,\eta):=\int_{S^1}(x,r)^*\lambda-\eta\int_{S^1}\kappa(t)H(r)dt-\int_{S^1}\beta_\nu(r)L_t(x,r)\,dt.
\end{equation}

A point $\left(u(t),\eta\right)\in M$ (where $u(t)=(x(t),r(t))$) is a critical point of $\mathcal{A}^{(\hat{\varphi},\nu)}_{(\alpha_0,\alpha_1)}$ if
\begin{align*}
\begin{cases}
	&\dot{u}(t)=\eta\kappa(t)X_{H}(u(t))+\beta_\nu(r(t))X_L(u(t))\\
	&\int_0^1\kappa(t)H(u(t))\,dt=0.
\end{cases}
\end{align*}

Note that it follows that for a critical point $H(u(t))=0$ for $t\in[0,\frac12]$ and thus $r(t)=1$ for $t\in[0,\frac12]$.

\begin{lem}\label{lem:Cadmissible}
 Assume that $\hat{\varphi}=\left\{\varphi_t\right\}_{t\in[0,1]}\in\mathcal{P}\Cont_0(\Sigma,\xi)$ is an admissible path of contactomorphisms and define 
\begin{equation}\label{eq:Cadmissible}
 C(\hat{\varphi};\alpha_1):=\max_{t\in[0,1]}\int_0^t\max_{x\in\Sigma}\left|\frac{\dot{\rho}_s(x)}{\rho_s(x)^2}\right|\,ds.
\end{equation}
Let $(x,r,\eta)\in\displaystyle{\crit\left(\mathcal{A}^{(\hat{\varphi},\nu)}_{(\alpha_0,\alpha_1)}\right)}$. If $\nu>C(\hat{\varphi};\alpha_1)$  then every critical point of $\mathcal{A}^{(\hat{\varphi},\nu)}_{(\alpha_0,\alpha_1)}$ has image contained in $\Sigma\times(e^{-\nu},e^{\nu})\times\R$, i.e. $r\left(S^1\right)\subset\left(e^{-\nu},e^{\nu}\right)$. 
\end{lem}

The proof of the lemma is analogous to the proof of \cite[Lemma 3.5]{AlbersFuchsMerry2013} after the observation that for $\nu>C(\hat{\varphi};\alpha_1)$ we have $r\left(S^1\right)\subset\left(e^{-\nu},e^{\nu}\right)$ and the fact that $\lambda=r\alpha_1$ in a neighborhood of all the critical points of the action functional $\mathcal{A}^{(\hat{\varphi},\nu)}_{(\alpha_0,\alpha_1)}$.

From now on, we fix $\nu>C(\hat{\varphi};\alpha_1)$ and  $\lambda \in \Omega_{\nu}(\alpha_0,\alpha_1)$ and write $\mathcal{A}^{\hat{\varphi}}_{(\alpha_0,\alpha_1)}:=\mathcal{A}^{(\hat{\varphi},\nu)}_{(\alpha_0,\alpha_1)}$.

	\begin{rem}\label{actionatcrit} Note that after the above choice of $\nu$, at a critical point $(u,\eta)$ it holds that
		\begin{align*}
		\mathcal{A}^{\hat{\varphi}}_{(\alpha_0,\alpha_1)}(u,\eta)
		&=\int_{S^1}\lambda(\dot{u}(t))\,dt-\int_{\frac12}^{1}L_t(u(t))\,dt\\
		&=\eta+\int_{S^1}\left(\lambda(X_L(u))-L_t(u)\right)\,dt\\
		&=\eta
		\end{align*}
		since $L_t=rl_t$ where $l_t$ is the contact Hamiltonian of the path $\hat{\varphi}$ with respect to $\alpha_1$.
		
	\end{rem}

\begin{defn}\label{def:nondegenerate}
 A path $\hat{\varphi}$ is \textbf{nondegenerate} if $\mathcal{A}^{\hat{\varphi}}_{(\alpha_0,\alpha_1)}:\mathcal{L}M\times\R\rightarrow\R$ is a Morse-Bott function which means
 $\crit\left(\mathcal{A}^{\hat{\varphi}}_{(\alpha_0,\alpha_1)}\right)\subset\mathcal{L}M$ is a submanifold and for each $(u,\eta)\in\crit\left(\mathcal{A}^{\hat{\varphi}}_{(\alpha_0,\alpha_1)}\right)$ we  have
 \[
	 T_{(u,\eta)}\crit\left(\mathcal{A}^{\hat{\varphi}}_{(\alpha_0,\alpha_1)}\right)=\ker\hess\left(\mathcal{A}^{\hat{\varphi}}_{(\alpha_0,\alpha_1)}\right)(u,\eta).
 \]
 see \cite{Frauenfelder2004} for more details.

\end{defn}

It is standard to show that nondegeneracy is a generic property for paths of contactomorphisms.

%

We want to choose $J$ an almost complex structure on $M$ in such a way that it satisfies the following properties.

\begin{defn}\label{SFT-like}
 Let $\alpha\in\mathcal{C}(\xi)$. We say that an almost complex structure $J$ is \textbf{SFT-like with respect to $\alpha$}, if 
 \begin{enumerate}
  \item[$\cdot)$] it is invariant under translations $(x,r)\mapsto (x,e^c r)$ for $c\in\R$,
  \item[$\cdot)$] it preserves $\xi$ and
  \item[$\cdot)$] satisifes $J R_\alpha=r\partial_r$,
 \end{enumerate}
 where $R_\alpha$ denotes the Reeb vector field wrt $\alpha$.
\end{defn}

Let $J:=\left\{J_{t}\right\}_{t\in S^1}$ be a family of almost complex structures compatible with $d\lambda$. With the sign conventions that we use this means that $d\lambda(J\cdot,\cdot)$ defines a family of Riemannian metrics on $M$. In the following, we always assume that $J$ is independent of $t$ outside a compact set and 
\begin{equation}\label{complex structure}
 J \text{ is SFT-like wrt } \alpha_0 \text{ on } \Sigma\times(0,e^{-4\nu}] \text{ and SFT-like wrt } \alpha_1 \text{ on } \Sigma\times[e^{2\nu},+\infty).
\end{equation}

Note that the set of almost complex structures of the form \eqref{complex structure} w.r.t. some $\alpha$ is connected.


For $(u,\eta)\in\mathcal{L}M\times\R$, let $\left\llangle\cdot,\cdot\right\rrangle_{J}$ on $T_{(u,\eta)}\left(\mathcal{L}M\times\R\right)$ denote the $L^2$-inner product defined by
\[
 \left\llangle(\hat{u},\hat{\eta}),(\hat{v},\hat{\tau})\right\rrangle_{J}:=\int_{S^1}d\lambda\left(J_{t}\hat{u},\hat{v}\right)dt+\hat{\eta}\hat{\tau},\quad(\hat{u},\hat{\eta}),(\hat{v},\hat{\tau})\in T_{(u,\eta)}\left(\mathcal{L}M\times\R\right).
\]

The gradient $\nabla_{J}\mathcal{A}^{\hat{\varphi}}_{(\alpha_0,\alpha_1)}(u,\eta)$ with respect to the above inner product is given by
\[
 \nabla_{J}\mathcal{A}^{\hat{\varphi}}_{(\alpha_0,\alpha_1)}(u,\eta)=\left(J_t(u)\left(\partial_t u-\eta\kappa X_H(u)-\frac{\partial\beta_\nu(r)}{\partial r}X_L(u)\right),-\int_{S^1}\kappa H(u)dt\right).
\]

We look at negative gradient flow lines of $\nabla_J\mathcal{A}^{\hat{\varphi}}_{(\alpha_0,\alpha_1)}(u,\eta)$, i.e. maps $(u,\eta)\in C^{\infty}(\R\times S^1,M)\times C^{\infty}(\R,\R)$ satisfying
\[
\partial_s(u,\eta)+\nabla_J\mathcal{A}^{\hat{\varphi}}_{(\alpha_0,\alpha_1)}(u,\eta)=0.
\]

Thus the Floer equations of $\mathcal{A}^{\hat{\varphi}}_{(\alpha_0,\alpha_1)}$ are given by
\begin{align*}
 \partial_s u +J_t(u)\left(\partial_t u-\eta \kappa X_H(u)-\frac{\partial\beta_\nu(r)}{\partial r}X_L(u)\right)&=0\\
 \partial_s \eta - \int_{S^1}\kappa H(u)dt&=0. 
\end{align*}

and have energy
\[
 E(u,\eta)=\int_{\R}\int_{S^1}\left|\partial_s(u,\eta)\right|_J^2\,dt\,ds.
\]

Let $a_-,a_+\in\R$. The moduli space $\mathcal{M}^{a_+}_{a_-}\left(\mathcal{A}^{\hat{\varphi}}_{(\alpha_0,\alpha_1)},J\right)$ is the set of all solutions $(u(s),\eta(s))$ of the Floer equations with 
\begin{align*}
	a_-\geq\lim_{s\rightarrow -\infty}\mathcal{A}^{\hat{\varphi}}_{(\alpha_0,\alpha_1)}(u(s),\eta(s))\quad\text{and}\quad \lim_{s\rightarrow +\infty}\mathcal{A}^{\hat{\varphi}}_{(\alpha_0,\alpha_1)}(u(s),\eta(s))\geq a_+.	
\end{align*}

Note that in this case the energy is precisely given by the difference of the action values. Thus, it actually holds that $a_-\geq a_+$ since solutions of the Floer equations with nonnegative energy must be decreasing. 

\subsection{Compactness of the moduli spaces}

\begin{thm}\label{compactness}
 Let $J$ be a family of almost complex structures compatible with $d\lambda$ that are both independent of $t$ and SFT-like outside of $\Sigma\times[e^{-4\nu},e^{2\nu}]$. Then the moduli spaces $\mathcal{M}^{a_+}_{a_-}\left(\mathcal{A}^{\hat{\varphi}}_{(\alpha_0,\alpha_1)},J\right)$ are compact in the $C^\infty_{\text{loc}}$-topology.
\end{thm}


The crucial property to achieve compactness is the fundamental lemma.

\begin{lem}\label{fundlemma}\cite[Proposition 3.2]{CieliebakFrauenfelder2009}
 There exist constants $C_0,C_1>0$ such that for any $(u,\eta)\in\mathcal{M}^{a_+}_{a_-}\left(\mathcal{A}^{\hat{\varphi}}_{(\alpha_0,\alpha_1)},J\right)$ we have that
 \begin{equation}\label{fundlemeq}
  \left\|\nabla_J\mathcal{A}^{\hat{\varphi}}_{(\alpha_0,\alpha_1)}(u,\eta)\right\|_J\leq C_0\quad \Rightarrow \quad  \left|\eta\right|\leq C_1\left(1+|\mathcal{A}^{\hat{\varphi}}_{(\alpha_0,\alpha_1)}(u,\eta)|\right).
 \end{equation}
\end{lem}

	The proof given in \cite{CieliebakFrauenfelder2009} goes still through. The proof uses the behaviour of flow lines in a neighborhood of the hypersurfaces. In our setting, a neighborhood of the hypersurface still looks the same apart from rescaling of the contact form.

We will show the following proposition from which Theorem \ref{compactness} follows immediately.

\begin{prop}\label{prop-compactness}
 In the setting of Theorem \ref{compactness} there exist $k,l>0$ such that
 \[
  Im(u)\subset \Sigma\times\left[k,l\right] \quad\text{for any } (u,\eta)\in\mathcal{M}^{a_+}_{a_-}\left(\mathcal{A}^{\hat{\varphi}}_{(\alpha_0,\alpha_1)},J\right).
 \]
\end{prop}

We need the notion of a trivial cylinder.
\begin{defn}\label{trivialcylinder}
A map $u:\,\R\times S^1\rightarrow \Sigma\times \R^+$ of the form $(\gamma(\pm tP),c e^{\pm Ps})$ for some $c\in\R^+$ and $j\partial_t=\partial_s$ for the complex structure $j$ on $\R\times S^1$ is called a trivial cylinder over a $P$-periodic Reeb orbit $\gamma$. Note that such a cylinder is a $J$-holomorphic map for any SFT-like $J$.
\end{defn}

Moreover, we need the definition of the Hofer energy.

\begin{defn}\label{Hoferenergy}
	Let $(Z,j)$ be a compact Riemann surface (possibly disconnected and with boundary). Let $u=(x,r):~Z\rightarrow M$ be a $(j,J)$-holomorphic map. The Hofer energy of a flow line $u$ is given by
	\[
		E_H(u)=\sup_{m\in \mathcal{S}}\int_Z u^*d(m\alpha)=\sup_{m\in\mathcal{S}}\left(\int_Z u^*(m d\alpha)+\int_Z u^*(m'(s)\,ds\wedge\alpha)\right)\in[0,+\infty],
	\]
	where $\mathcal{S}:=\left\{\left.m\in C^\infty\left(\R,[0,1]\right)\right| ~m'\geq 0 \right\}$.
\end{defn}

To prove Proposition \ref{prop-compactness} we use the following theorem from \cite{AlbersFuchsMerry2013}  which is the special case of the SFT-compactness results that we need.

\begin{thm}\label{SFT}\cite[Theorem 5.3]{AlbersFuchsMerry2013}
 Let $\left(M,\lambda\right)$ be as before. Suppose $(Z_k,j_k)$ is a family of compact (possibly disconnected) Riemann surfaces with boundary and uniformly bounded genus. Assume that
 \[
  u_k=(y_k,a_k):~Z_k\rightarrow \Sigma\times\R^{+}=M
 \]
is a sequence of $(j_k,J)$-holomorphic maps with $E_H(u_k)<K$ for some $K>0$ and which are nonconstant on each connected component of $Z_k$ and satisfy $a_k(\partial_k Z_k)\subset[D,+\infty)$, where $D<e^{-4\nu}$. Moreover, assume that $\inf_k \inf_{Z_k} a_k=0$. Then there exists a subsequence $k_n$ and cylinders $C_n\subset Z_{k_n}$ biholomorphically equivalent to standard cylinders $S^1\times[-L_n,+L_n]$ such that $L_n\rightarrow +\infty$ and such that $\left.u_{k_n}\right|_{C_n}$ converges (up to an $\R$-shift) in $C^{\infty}_{\text{loc}}\left(S^1\times\R,\Sigma\times\R^+\right)$ to a trivial cylinder over a Reeb orbit of $\alpha_0$ with period $\leq K$.
\end{thm}

%

\begin{proof}[Proof of Proposition \ref{prop-compactness}]
First, we note that on $\Sigma\times(e^{2\nu},+\infty)$ the Hamiltonian vector fields of $H$ and $L$ vanish and thus solutions of the Floer equations are in fact $J$-holomorphic curves. Hence, we can apply the Maximum principle to keep Floer trajectories from escaping to $+\infty$ and consequently there is $l>0$ such that $Im(u)\subset \Sigma\times(0,l]$ for all $(u,\eta)\in \mathcal{M}\left(\mathcal{A}^{\hat{\varphi}}_{(\alpha_0,\alpha_1)},J\right)$.

Moreover, we observe that for any $(u,\eta)\in\mathcal{M}^{a_+}_{a_-}\left(\mathcal{A}^{\hat{\varphi}}_{(\alpha_0,\alpha_1)},J\right)$ the restriction $\left.u\right|_{u^{-1}\left(\Sigma\times(0,e^{-4\nu}]\right)}$ is a $J$-holomorphic map.

\textit{Claim:}	 For $(u,\eta)\in\mathcal{M}_{a_-}^{a_+}\left(\mathcal{A}_{(\alpha_0,\alpha_1)}^{\hat{\varphi}};J\right)$ the Hofer energy $E_H(u_k)$ is uniformly bounded by $e^{4\nu}(a_--a_+).$

The proof of this claim can be found in \cite[Proof of Theorem 3.9]{AlbersFuchsMerry2013}. 
For the convenience of the reader, we include it here. Indeed, we can estimate
\begin{align*}
a_--a_+&\geq E(u,\eta)=\int_{-\infty}^{+\infty}\left\|\nabla\mathcal{A}_{(\alpha_0,\alpha_1)}^{\hat{\varphi}}(u,\eta)\right\|^2_J\,ds
=\int_{-\infty}^{+\infty}\left\|\partial_s(u,\eta)\right\|^2_J\,ds\\
&\geq\int_{-\infty}^{+\infty}\int_{S^1}d\lambda(J\partial_s u,\partial_s u)\,dt\,ds\geq
\int_{u^{-1}(\Sigma\times(0,e^{-4\nu}))}d\lambda(J\partial_s u,\partial_s u)\,dt\,ds\\
&=\int_{u^{-1}(\Sigma\times(0,e^{-4\nu}))} u^*d\lambda,
\end{align*}
where on this domain $d\lambda=d(r\epsilon\alpha_0)$. Here we used that $u$ restricted to $u^{-1}(\Sigma\times(0,e^{-4\nu}))$ is $J$-holomorphic. 

On the other hand we can estimate for $m\in S$
\begin{align*}
\int_{u^{-1}(\Sigma\times(0,e^{-4\nu}))} u^*d(m\epsilon\alpha_0)
&\,\,\,\leq\int_{u^{-1}(\Sigma\times(0,e^{-4\nu}))} u^*d(\epsilon\alpha_0)\\
&\overset{\text{Stokes}}{=}e^{4\nu}\int_{u^{-1}(\Sigma\times(0,e^{-4\nu}))} u^* d\lambda
\end{align*}
which concludes the proof of the claim.

 Assume now by contradiction that there is no $k>0$ such that $Im(u)\subset[k,l]$ and thus there exists a sequence $u_k=(y_k,a_k)$ such that $\lim_k\inf_{Z_k}a_k=0$. Choose $T<e^{-4\nu}$ such that $T$ is a regular value for all $a_k$'s. Let $Z_k:=(u_k)^{-1}(\Sigma\times(0,T])$ and consider the $J$-holomorphic curves $v_k:=\left.u_k\right|_{Z_k}$. Since for each $k$ the $v_k$ are gradient flow lines of $\mathcal{A}^{\hat{\varphi}}_{(\alpha_0,\alpha_1)}$ and its asymptotes are critical points $(u,\eta)$ of $\mathcal{A}^{\hat{\varphi}}_{(\alpha_0,\alpha_1)}$ where $u$ is contained in $\Sigma\times(e^{-\nu},e^{\nu})$ the $Z_k$'s are compact possibly disconnected Riemann surfaces of genus $0$. Since we choose $T<e^{-4\nu}$ and so that $T$ is a regular value, no Floer cylinder is constant and thus $v_k$ has no constant components. 
 By choice of $T$, 
 $v_k$ also satisfies $a_k(\partial Z_k)\subset[T,+\infty)$. 
 Thus $v_k$ satisfies all the assumptions of Theorem \ref{SFT} and hence there exists a subsequence $v_{k_n}\rightarrow v_{k_0}$ whose restriction to cylinders converges to a trivial cylinder over a Reeb orbit of $\alpha_0$ since we have $a_k\rightarrow  0 $. 
 
 Thus there is an embedded circle $S$ in the domain $S^1\times\R$ of $v_{k_0}$ such that the restriction of $y_{k_0}$ to $S$ is homotopic to a Reeb orbit $\gamma$ of $\alpha_0$. The domain of $y_{k_0}$ is $S^1\times \R$ and thus either $S$ is a circle bounding a disk or $S$ is a circle of the form $S^1\times{s}$. In the first case it is clear that $S$ and thus $y_{k_0}$ restricted to $S$ is contractible where in the latter case it follows that the image of $S$ under $y_{k_0}$ is contractible since it is homotopic to the asymptotic end of the cylinder which is contractible since the asymptotic ends lie in $\mathcal{L}M$.
 	
 Finally, we have shown that $y_k$ converges to a contractible Reeb orbit of $\alpha_0$ which is a contradiction since $\alpha_0$ is without contractible Reeb orbits.\\
\end{proof}

\section{Definition of Rabinowitz Floer homology}

Here we give a sketch of the definition of Rabinowitz Floer homology in our setting. For details see for example \cite{CieliebakFrauenfelder2009} or \cite{AlbersFrauenfelder2010c}.
If the functional $\mathcal{A}^{\hat{\varphi}}_{(\alpha_0,\alpha_1)}$ is Morse, which is a property that does not depend on $\lambda \in \Omega_{\nu}(\alpha_0, \alpha_1)$, because nondegeneracy is a local condition, we associate to it the Rabinowitz Floer homology groups. Since Morse is a generic property, see \cite{AlbersFrauenfelder2010c}, one can use invariance under perturbations,  as proved in Section \ref{sec:invariance}, to extend the definition to all functionals $\mathcal{A}^{\hat{\varphi}}_{(\alpha_0,\alpha_1)}$ as follows: any path of contactomorphisms can be written as the limit of nondegenerate paths since those form a residual set. We use this sequence to define the Rabinowitz Floer homology of a degenerate path as the limit of the homology of the functionals associated to a sequence of nondegenerate paths. Invariance of Rabinowitz Floer homology shows that the homology does not depend on the chosen sequence.

Note that the choice $\hat{\varphi}=\id$ is not generic. Whilst the functional $\mathcal{A}^{\id}_{(\alpha_0,\alpha_1)}$ is not Morse it satisfies the Morse-Bott conditions. Thus we will explain here how to construct the Rabinowitz Floer complex under the assumption that $\mathcal{A}^{\hat{\varphi}}_{(\alpha_0,\alpha_1)}$ is Morse-Bott.

Assume therefore that the functional is Morse-Bott. Choose a Morse function $f$ and a Riemannian metric $g$ on the critical submanifold $\crit\left(\mathcal{A}^{\hat{\varphi}}_{\left(\alpha_0,\alpha_1\right)}\right)$. Denote by $\crit\left(f\right)$ the set of critical points of $f$. Let $\CF\left(\mathcal{A}^{\hat{\varphi}}_{\left(\alpha_0,\alpha_1\right)}\right)$ be the $\Z_2$-vector space consisting of formal sums $\sum_{w\in\crit(f)}n_w\,w$, where the coefficients $n_w\in\Z_2$ satisfy
\[
 \#\left\{\left.w\in\crit(f)\right|\,n_w\neq 0,\,\mathcal{A}^{\hat{\varphi}}_{\left(\alpha_0,\alpha_1\right)}(w)\leq \kappa\right\}<\infty,
 \]
for every $\kappa\in\R$.

Assume that $(f,g)$ is a Morse-Smale pair. Then, for any $z,w\in\crit(f)$ we denote by $\mathcal{M}\left(z,w;\mathcal{A}^{\hat{\varphi}}_{\left(\alpha_0,\alpha_1\right)},f,J,g\right)$ the moduli space of gradient trajectories with cascades from $z$ to $w$ with respect to the Riemannian metric $g$ on $\crit\left(\mathcal{A}^{\hat{\varphi}}_{\left(\alpha_0,\alpha_1\right)}\right)$, and we denote by $\mathcal{M}^0\left(z,w;\mathcal{A}^{\hat{\varphi}}_{\left(\alpha_0,\alpha_1\right)},f,J,g\right)$ its zero-dimensional part. The compactness result, Theorem \ref{compactness}, shows that if $|z|-|w|=1$, then $\mathcal{M}\left(z,w;\mathcal{A}^{\hat{\varphi}}_{\left(\alpha_0,\alpha_1\right)},f,J,g\right)$ is compact and thus its zero-dimensional part is a finite set.

We define the boundary operator
\[
\partial:\,\CF\left(\mathcal{A}^{\hat{\varphi}}_{\left(\alpha_0,\alpha_1\right)},f;J,g\right)\rightarrow\CF\left(\mathcal{A}^{\hat{\varphi}}_{\left(\alpha_0,\alpha_1\right)},f;J,g\right)
\]
as the linear extension of
\[
\partial(z)=\sum_{w\in\crit(f)}n(z,w)\,w,
\]
where $z\in\crit(f)$ and $n(z,w)=\#\mathcal{M}^0\left(z,w;\mathcal{A}^{\hat{\varphi}}_{\left(\alpha_0,\alpha_1\right)},f,J,g\right) \mod 2\in\Z_2$.
Standard arguments in Floer theory also yield that $\partial^2=0$. Finally, we can define the Floer homology groups of the above constructed chain complex
\[		 
	\HF\left(\CF\left(\mathcal{A}^{\hat{\varphi}}_{\left(\alpha_0,\alpha_1\right)},f;J,g\right),\partial\right)
	:=\h\left(\CF\left(\mathcal{A}^{\hat{\varphi}}_{\left(\alpha_0,\alpha_1\right)},f;J,g\right),\partial\right).
\]
Standard arguments show that these do not depend on the choice of $f,g$ and $J$ and thus we define the Rabinowitz Floer homology of $(M,\lambda)$ as
\[
 \RFH\left((M,\lambda);\hat{\varphi}\right):=\HF\left(\CF\left(\mathcal{A}^{\hat{\varphi}}_{\left(\alpha_0,\alpha_1\right)},f;J,g\right),\partial\right).
\]

See \cite{CieliebakFrauenfelder2009} for more details on the definition of Rabinowitz Floer homology.

\section{Continuation maps}\label{sec:continuationmaps}

In the next step we show that the above defined groups are independent of $\alpha_1$. We construct an isomorphism between
\[
 \widetilde{\Phi}: \RFH\left((M,\lambda_{1});\hat{\varphi}\right)\rightarrow\RFH\left((M,\lambda_2);\hat{\varphi}\right)
\]
for any $\alpha_1,\alpha_2\in\mathcal{C}(\xi)$, and 
$\lambda_1 \in \Omega_{\nu}(\alpha_0,\alpha_1), \lambda_2 \in \Omega_{\nu}(\alpha_0, \alpha_2)$ for  $\nu>\max\left\{C(\hat{\varphi};\alpha_1),C(\hat{\varphi};\alpha_2)\right\}$.

In order to do so, we define a homotopy $\lambda_s$ between $\lambda_1$ and $\lambda_2$. 
Let $\zeta\in C^{\infty}\left(\R,[0,1]\right)$ with $0\leq\dot{\zeta}(s)\leq 2$ and such that
\[
\zeta(s)=
 \begin{cases}
  1,\quad&\text{for }s\geq 1\\
  0,\quad&\text{for }s\leq 0.
 \end{cases}
\]
Define 
\begin{equation}\label{contactsform}
	\lambda_s:=\lambda_1+\zeta(s)\left(\lambda_2-\lambda_1\right).
\end{equation}
Note that since $\partial_r f_{i}>0$, where $\lambda_i=f_i\alpha_0,\,i=1,2$, we still have that $d\lambda_s$ is a nondegenerate symplectic form. Note that this symplectic form $d\lambda_s$ is independent of $s$ for $s\notin[0,1]$ and that $\alpha_s:= \lambda_s|_{\Sigma \times \{1\}}$ is a supporting contact form for $\xi$.
                                                                              
Let $J_{1}$ and $J_{2}$ denote the $d\lambda_1$-, resp. $d\lambda_2$-, compatible families of SFT-like almost complex structures defined as in \eqref{complex structure}. Moreover let $J_s$ be such that
\[
 J_{s}=\begin{cases}
       J_{1},\quad\text{for }s\geq 1\\
       J_{2},\quad\text{for }s\leq 0
      \end{cases}
\]
where $J_{s}$ is a $d\lambda_s$-compatible almost complex structure and independent of $t$ outside a compact set and with $J_{1}$, $J_{2}$ SFT-like as defined in \eqref{complex structure}. That is $J_{1}$ is SFT-like wrt $\alpha_0$ on $\Sigma\times(0,e^{-4\nu}]$ and SFT-like wrt $\alpha_1$ on $\Sigma\times[e^{2\nu},+\infty)$ and $J_{2}$ analogously for $\alpha_0$ and $\alpha_2$.


Look at the $s$-dependent functional
\[
 \mathcal{A}_{\lambda_s}(u,\eta):=\int_{S^1}u^*\lambda_s-\eta\int_{S^1}\kappa(t)H(r)dt-\int_{S^1}\beta_\nu(r)L_t(u)\,dt,
\]
where $u=(x,r):S^1\rightarrow M$. The corresponding Floer equations are
\begin{align*}
 \partial_s u + J_{s,t}\left(\partial_t u-\eta \kappa X^{\lambda_s}_H(u)-\beta_\nu(r)X^{\lambda_s}_L\right)&=0\\
 \partial_s \eta + \int_{S^1} \kappa H(u)dt&=0,
\end{align*}
where, $u=u(s,t):\,\R\times S^1\rightarrow M$ and $\eta=\eta(s):\,\R\rightarrow\R$.
	Note that now, the Hamiltonian vector fields of $H$ and $L$ with respect to $\left.\lambda_s\right|_{\Sigma\times \left\{1\right\}}$ depend on $s$ since $\left.\lambda_s\right|_{\Sigma\times \left\{1\right\}}=\alpha_s$ is an $s$-dependent contact form.
	
From now on, we will denote $u(s):=u(s,\cdot):\,\R\times S^1\rightarrow M$.
Let $a_-,a_+\in\R$. Denote as before by $\mathcal{M}_{a_-}^{a_+}\left(\mathcal{A}_{\lambda_s},J_s\right)$ the solutions of the corresponding Floer equations of $\mathcal{A}_{\lambda_s}(u(s),\eta(s))$ whose asymptotes have action bounded by $a_{\pm}$
	\[
		a_-\geq \lim_{s\rightarrow -\infty}\mathcal{A}_{\lambda_s}(u(s),\eta(s))\quad\text{and}\quad 
		\lim_{s\rightarrow +\infty}\mathcal{A}_{\lambda_s}(u(s),\eta(s))\geq a_+.
	\]


\subsection{Compactness for the $s$-dependent moduli spaces}

\begin{thm}\label{continuationcompactness}
Assume $\lambda_s$ is a homotopy of 1-forms as defined in \eqref{contactsform}. There exists $\epsilon>0$ such that if $\sup_{\Sigma \times (0,e^{2\nu}]}\left\|\partial_s \lambda_s\right\|_\infty<\epsilon$ then the moduli space $\mathcal{M}^{a_+}_{a_-}\left(\mathcal{A}_{\lambda_s},J_s\right)$ is relatively compact in the $C^\infty_{\text{loc}}$-topology.

\end{thm}

Recall that $\left\|\partial_s \lambda_s\right\|_\infty$ is only nonzero for $s\in[0,1]$ since outside of $[0,1]$  $\lambda_s$ is independent of $s$.

The main task to achieve compactness is to bound the energy. We start with the observation that Lemma \ref{fundlemma} remains true in the s-dependent case.

\begin{lem}\label{sfundlemma}
 There exist constants $C_0,C_1>0$ such that for any $(u,\eta)\in\mathcal{M}^{a_+}_{a_-}\left(\mathcal{A}_{\lambda_s},J_s\right)$ we have that
 \begin{equation}\label{fundbound}
  \left\|\nabla_{J_s}\mathcal{A}_{\lambda_s}(u,\eta)\right\|_{J_s}\leq C_0\quad \Rightarrow \quad  \left|\eta\right|\leq C_1\left(1+\left|\mathcal{A}_{\lambda_s}(u,\eta)\right|\right).
 \end{equation}
\end{lem}
\begin{proof}
We can apply Lemma \ref{fundlemma} for any of the functionals $\mathcal{A}_{\lambda_s}$.
The constants $C_{0,s}$ and $C_{1,s}$ in \eqref{fundlemeq} for a fixed functional $\mathcal{A}_{\lambda_s}$ depend continuously on $s$. Take $C_0 := \min\{C_{0,s} \, |\, s \in [0,1]\}$  and $C_1 := \max\{C_{1,s} \, |\, s \in [0,1]\}$. 
\end{proof}

The following lemma is the main content of this paper and uses ideas of Bae-Frauenfelder, see \cite[Lemma 2.9; Theorem 2.10]{BaeFrauenfelder2010}. It establishes uniform bounds on $\|\eta(s)\|_{\infty}$ and the energy of Floer trajectories in the s-dependent case provided that the homotopy is sufficiently slow.

\begin{lem}\label{energybound}
Let $C_0$ and $C_1$ be chosen as required in Lemma \ref{sfundlemma} and fix $a_-,a_+\in\R$. Then, there is a constant $\rho = \rho_{\{\lambda_s\}}>0$ such that if $\sup_{s\in [0,1]}\sup_{\Sigma \times (0,e^{2\nu}]}  \left\|\partial_s \lambda_s\right\|_{J_s}<\rho$  the following holds:
There exist constants $C_\eta,C_E>0$ such that for any $(u(s),\eta(s))\in\mathcal{M}^{a_+}_{a_-}\left(\mathcal{A}_{\lambda_s},J_s\right)$
 \[
   \left\|\eta(s)\right\|_{\infty}<C_\eta \quad\text{and}\quad E(u(s),\eta(s))< C_E.
  \] 
\end{lem}

Using the above estimates we can prove Theorem \ref{continuationcompactness}.

\begin{proof}[Proof of Theorem \ref{continuationcompactness}]
The uniform energy bound of Lemma \ref{energybound} and the fact that $X_H$ and $X_L$ have support in $\Sigma\times (e^{-2\nu},e^{2\nu})$ as well as the fact that $\lambda_s$ restricted to $\Sigma\times (0,e^{-4\nu}]$ is just some multiple of $r\alpha_0$, respectively $r\alpha_1$ and $r\alpha_2$ on $\Sigma\times\left(e^{+3\nu},+\infty\right)$, shows that outside of $\Sigma\times (e^{-4\nu},e^{2\nu}]$ solutions of the Floer equations are exactly $J$-holomorphic curves. This allows us to apply the maximum principle to keep Floer trajectories from escaping to $+\infty$. As in the proof of Theorem \ref{compactness} we use Theorem \ref{SFT} to argue that neither do Floer trajectories escape to the negative end.

This shows that $u(s)$ stays in a compact subset of $\Sigma \times (0,+\infty)$, say $\Sigma\times[k,l]$. With the uniform $L^{\infty}$-bound for $\eta(s)$ from Lemma \ref{energybound} the usual bubbling-off arguments apply, thus yielding $L^\infty$-bounds for the derivatives of $u(s)$. Thus we can use Arzel\`{a}-Ascoli for families of Rabinowitz Floer trajectories whose asymptotes have action bounded by $a_\pm$. We finally get a $C_{\text{loc}}^{\infty}\left(\R \times S^1, M\right)\times C^{\infty}_{\text{loc}}\left(\R,\R\right)$ convergence of a subsequence of trajectories. The result then follows by the usual arguments.  
\end{proof}

It remains to prove Lemma \ref{energybound}.

\begin{proof}[Proof of Lemma \ref{energybound}]
  In the following we will denote 
\[
	\sup_{s \in[0,1]}\sup_{\Sigma \times (0,e^{2\nu}]}\left\|\partial_s \lambda_s\right\|_{J_s}=:\widetilde{C}. 
\]
  
  We want to estimate the energy
  \begin{align*}
   E(u(s),\eta(s))&=\int_{\R}\left\|\partial_s(u,\eta)\right\|_{J_s}^2\,ds=\int_{-\infty}^{+\infty}\left\|-\nabla_{J_s}\mathcal{A}_{\lambda_s}(u(s),\eta(s))\right\|_{J_s}^2\,ds\\
   &=\int_{-\infty}^{+\infty}\left\llangle-\nabla_{J_s}\mathcal{A}_{\lambda_s}(u(s),\eta(s)),-\nabla_{J_s}\mathcal{A}_{\lambda_s}(u(s),\eta(s))\right\rrangle_{J_s}\,ds\\
   &=\int_{-\infty}^{+\infty}-d\mathcal{A}_{\lambda_s}(u(s),\eta(s))\left(-\nabla_{J_s}\mathcal{A}_{\lambda_s}(u(s),\eta(s))\right)\,ds\\
   &=\int_{-\infty}^{+\infty}-d\mathcal{A}_{\lambda_s}(u(s),\eta(s))(\partial_s (u(s),\eta(s)))\,ds\\
   &=\int_{-\infty}^{+\infty}\left[-\frac{d}{d_s}\left(\mathcal{A}_{\lambda_s}(u(s),\eta(s))\right)+\left(\frac{\partial}{\partial_s}\mathcal{A}_{\lambda_s}\right)(u(s),\eta(s))\right]\,ds\\   
      &=\mathcal{A}_{\lambda_s}\left(u_{-},\eta_{-}\right)-\mathcal{A}_{\lambda_s}\left(u_{+},\eta_{+}\right)+\int_{0}^{1}\int_{S^1}(u(s))^*(\partial_s\lambda_s)\,ds\\
   &\leq a_--a_+ +\int_{0}^{1}\int_{S^1}(u(s))^*(\partial_s\lambda_s)\,ds.
  \end{align*}

Recall that by the maximum principle $u(s) \subset \Sigma \times (0,e^{2\nu}]$.
  Using the Floer equation for $\partial_s u(s)$ we have
  \begin{align*}
   \left|\int_{0}^{1}\int_{S^1}(u(s))^*(\partial_s\lambda_s)\,ds\right| 
   &\leq
   \int_0^1\sup_{s\in[0,1]}\sup_{\Sigma\times(0,e^{2\nu}]}\left\|\partial_s \lambda_s\right\|_{J_s}\cdot \int_{S^1}\left|\partial_t u(s)\right|\,dt \,ds\\
   &= \widetilde{C}\int_0^1 \left\|\partial_t u(s)\right\|_{J_s}\,ds
   =\widetilde{C}\int_0^1\left\|J_s\partial_s u(s)+\eta(s) \kappa X^{\lambda_s}_H+\beta_\nu X^{\lambda_s}_L\right\|_{J_s}\,ds\\
   &\leq\widetilde{C}\int_0^1\left\|\partial_s u(s)\right\|_{J_s}\,ds+\widetilde{C}\int_0^1\left\|\eta(s) \kappa X^{\lambda_s}_H+\beta_\nu X^{\lambda_s}_L\right\|_{J_s}\,ds\\
   &\overset{(\ast)}{\leq}\widetilde{C}\left[\left(\int_0^1\left\|\partial_s u(s)\right\|_{J_s}\,ds\right)^2+1\right]+\widetilde{C}\int_0^1\left\|\eta(s)\kappa X^{\lambda_s}_H+\beta_\nu X^{\lambda_s}_L\right\|_{J_s}\,ds\\
   &\leq \widetilde{C}\int_0^1\left\|\partial_s u(s)\right\|_{J_s}^2\,ds+\widetilde{C}+\widetilde{C}\int_0^1\left\|\eta(s) \kappa X^{\lambda_s}_H+\beta_\nu X^{\lambda_s}_L\right\|_{J_s}\,ds\\
   &=\widetilde{C}\left(E(u(s),\eta(s))+1+  \left\|\eta\right\|_{\infty} \left\|\kappa X^{\lambda_s}_H\right\|_{\infty}+ \left\|\beta_\nu X^{\lambda_s}_L\right\|_{\infty}\right)\\
   &\leq\widetilde{C}\left(E(u(s),\eta(s))+1+ \left\|\eta\right\|_{\infty} \left\|X^{\lambda_s}_H\right\|_{\infty}+ \left\|X^{\lambda_s}_L\right\|_{\infty}\right),
  \end{align*}
	where we wrote $\|X^{\lambda_s}_H\|_{\infty} := \sup_{s\in[0,1]}\sup_{M}\|X_{H_s}\|_{J_s}$ (and analogously for $\|X^{\lambda_s}_L\|_{\infty}$) and where $(\ast)$ used the fact that for $x\geq 0$ it holds that $x\leq x^2+1$.
	
  Putting everything together yields
  \[
   E(u(s),\eta(s))\leq a_--a_+ + \widetilde{C}\left(E(u(s),\eta(s))+1+\left\|\eta\right\|_{\infty}\left\|X^{\lambda_s}_H\right\|_{\infty}+\left\|X^{\lambda_s}_L\right\|_{\infty}\right)
  \]
  and thus
  \[
   \left(1-\widetilde{C}\right)E(u(s),\eta(s))\leq a_--a_+ +\widetilde{C}\left(1+\left\|\eta\right\|_{\infty}\left\|X^{\lambda_s}_H\right\|_\infty+\left\|X^{\lambda_s}_L\right\|_{\infty}\right).
  \]
  So for $\widetilde{C}<1$ we have that
  \[
    E(u(s),\eta(s))\leq \frac{a_--a_+}{1-\widetilde{C}} +\frac{\widetilde{C}}{1-\widetilde{C}}\left(1+\left\|\eta\right\|_{\infty}\left\|X^{\lambda_s}_H\right\|_\infty+\left\|X^{\lambda_s}_L\right\|_{\infty}\right)
   \]
    and thus if $\widetilde{C}<\frac12$ also
    \begin{equation}\label{eq:Ebound}
     E(u(s),\eta(s))<2\left|a_--a_+\right|+2\widetilde{C}\left(1+\left\|\eta\right\|_{\infty}\left\|X^{\lambda_s}_H\right\|_\infty+\left\|X^{\lambda_s}_L\right\|_{\infty}\right). 
  \end{equation}

 Note also that 
 \begin{align}
  \left|\mathcal{A}_{\lambda_s}(u(s),\eta(s))\right|&\leq\max\left\{|a_+|,|a_-|\right\}+\int_{-\infty}^{+\infty}\left|\frac{\partial}{\partial s}\Big(\mathcal{A}_{\lambda_s}(u(s),\eta(s))\Big)\right|\,ds \nonumber \\
  &=\max\left\{|a_+|,|a_-|\right\}+\int_0^1\left|\int_{S^1}(u(s))^*(\partial_s\lambda_s)\right|\,ds \label{eq:A} \\
  &\leq\max\left\{|a_+|,|a_-|\right\} +\widetilde{C}\left(E(u(s),\eta(s))+1+\left\|\eta\right\|_\infty\left\|X^{\lambda_s}_H\right\|_\infty+\left\|X^{\lambda_s}_L\right\|_\infty\right) \label{eq:boundforA}
 \end{align}
 
	Indeed, to show \eqref{eq:A}, we note that
	\begin{align*}
	&\mathcal{A}_{\lambda_s}(u(s),\eta(s))=\underbrace{\mathcal{A}_{\lambda_-}(u_{-},\eta_{-})}_{\leq a_{-}}+\int_{-\infty}^s\frac{\partial}{\partial s}\left(\mathcal{A}_{\lambda_s}(u(s),\eta(s))\right)\,ds \quad\text{and}\\
	&\mathcal{A}_{\lambda_s}(u(s),\eta(s))=\underbrace{\mathcal{A}_{\lambda_+}(u_{+},\eta_{+})}_{\geq a_{+}}-\int_{s}^{+\infty}\frac{\partial}{\partial s}\left(\mathcal{A}_{\lambda_s}(u(s),\eta(s))\right)\,ds,
	\end{align*}
	where
	\[
	\frac{\partial}{\partial s}\left(\mathcal{A}_{\lambda_s}(u(s),\eta(s))\right)=\left(\frac{\partial}{\partial s}\mathcal{A}_{\lambda_s}\right)(u(s),\eta(s))+\underbrace{d\mathcal{A}_{\lambda_s}(\partial_s u(s),\partial_s\eta(s))}_{=-\left\|\nabla_{J_s}\mathcal{A}_{\lambda_s}(u(s),\eta(s))\right\|_{J_s}^2}.
	\]
	But in fact the homotopy, $\lambda_s$ only changes on $[0,1]$ and so we have that
	\begin{align*}
	\int_{-\infty}^s\left(\frac{\partial}{\partial s}\mathcal{A}_{\lambda_s}\right)(u(s),\eta(s))\,ds
	&=\int_0^s\left(\frac{\partial}{\partial s}\mathcal{A}_s\right)(u(s),\eta(s))\,ds
	\leq\int_0^1\left|\left(\frac{\partial}{\partial_s}\mathcal{A}_s\right)(u(s),\eta(s))\right|\,ds\\
	&=\int_0^1\left|\int_{S^1}(u(s))^*(\partial_s\lambda_s)\right|\,ds
	\end{align*}
	and analogously for $\int_s^{+\infty}$. Finally, we have
	\begin{align*}
	\mathcal{A}_{\lambda_s}(u(s),\eta(s))&\leq a_{-}+\int_{-\infty}^s\left(\frac{\partial}{\partial s}\mathcal{A}_{\lambda_s}\right)(u(s),\eta(s))-\left\|\nabla\mathcal{A}_{\lambda_s}(u(s),\eta(s))\right\|^2\\
	&\leq a_{-}+\int_0^1\left|\int_{S^1}(u(s))^*(\partial_s\lambda_s)\right|\,ds
	\end{align*}
	and
	\begin{align*}
	\mathcal{A}_{\lambda_s}(u(s),\eta(s))&\geq a_{+}-\int_{s}^{+\infty}\left(\frac{\partial}{\partial s}\mathcal{A}_{\lambda_s}\right)(u(s),\eta(s))+\left\|\nabla\mathcal{A}_{\lambda_s}(u(s),\eta(s))\right\|^2\\
	&\geq a_{+}-\int_0^1\left|\int_{S^1}(u(s))^*(\partial_s\lambda_s)\,dt\right|\,ds.
	\end{align*}
	This gives \eqref{eq:A}
	\[
	\left|\mathcal{A}_{\lambda_s}(u(s),\eta(s))\right|\leq \max\left\{|a_{-}|,|a_{+}|\right\}+\int_0^1\left|\int_{S^1}(u(s))^*(\partial_s\lambda_s)\right|\,ds.
	\]

Now define for $s\in\R$
\[
 \tau(s):=\inf\left\{\tau\geq 0:\quad\left\|\nabla_{J_s}\mathcal{A}_{\lambda_s}\Big(u(s+\tau),\eta(s+\tau)\Big)\right\|_{J_s}<C_0 \right\},
\]
where $C_0$ is the constant from Lemma \ref{sfundlemma}.
Then $\tau(s)$ satisfies
\begin{equation}\label{eq:tau}
 \tau(s)\leq\frac{E(u(s),\eta(s))}{C_0 ^2}.
\end{equation}
Indeed,
\[
 E(u(s),\eta(s))\geq\int_{s}^{s+\tau(s)}\left\|\nabla_{J_s}\mathcal{A}_{\lambda_s}(u(s),\eta(s))\right\|_{J_s}^2\,ds\geq \tau(s)C_0^2.
\]

Hence, using the estimate \eqref{fundbound} from Lemma \ref{sfundlemma} and H\"older's inequality we have

\begin{align*}
 \left|\eta(s)\right|&~=\left|\eta(s+\tau(s))-\int_s^{s+\tau(s)}\partial_s\eta(s)\,ds\right|
 \leq \left|\eta(s+\tau(s))\right|+\int_{s}^{s+\tau(s)}\left|\partial_s\eta(s)\,ds\right|\,ds\\
 &~\leq C_1\left(\left|\mathcal{A}_{\lambda_s}(u(s),\eta(s))\right|+1\right)+\left|\tau(s)\right|^{\frac12}\left(\int_s^{s+\tau(s)}\left|\partial_s\eta(s)\right|^2\right)^{\frac12}.
 \end{align*}
 We can estimate
 \[
 \left(\int_s^{s+\tau(s)}\left|\partial_s\eta(s)\right|^2\right)^{\frac12}
 \leq\left( \int_s^{s+\tau(s)}\left\|\partial_s (u(s),\eta(s))\right\|_{J_s}^2\,ds \right)^{\frac12} \leq \left(E(u(s),\eta(s))\right)^{\frac12}
 \]
 and use equations \eqref{eq:boundforA} and \eqref{eq:tau} to deduce that
 \begin{align*}
 \left|\eta(s)\right|&\overset{\eqref{eq:boundforA},\eqref{eq:tau}}{\leq}C_1\left[\max\left\{|a_+|,|a_-| \right\}+\widetilde{C}\left(E(u(s),\eta(s))+1+\left\|\eta\right\|_\infty\left\|X^{\lambda_s}_H\right\|_\infty+\left\|X^{\lambda_s}_L\right\|_\infty\right)\right]\\
 &\hspace{20ex}+\frac{E(u(s),\eta(s))}{C_0}.
 \end{align*}

Inserting the above in \eqref{eq:Ebound} yields
\begin{align*}
 \left\|\eta\right\|_\infty & \leq  C_1\left(1+\max\left\{|a_+|,|a_-|\right\}\right)\\
 &\quad +\left(2\left|a_--a_+\right|+2\widetilde{C}\left(1+\left\|\eta\right\|_{\infty}\left\|X^{\lambda_s}_H\right\|_\infty+\left\|X^{\lambda_s}_L\right\|_{\infty}\right)\right)\left(C_1\widetilde{C}+\frac{1}{C_0}\right)\\
 &\quad +C_1\widetilde{C}\left(1+\left\|\eta\right\|_\infty\left\|X^{\lambda_s}_H\right\|_\infty+\left\|X^{\lambda_s}_L\right\|_\infty\right).
\end{align*}
Hence we have that
\begin{align*}
 &\left(1-\widetilde{C}\left\|X^{\lambda_s}_H\right\|_{\infty}\left(2C_1\widetilde{C}+\frac{2}{C_0}+C_1\right)\right)\left\|\eta\right\|_\infty\\
 &\qquad\leq C_1\left(1+\max\left\{|a_+|,|a_-|\right\}\right)
 +\left(2\left|a_--a_+\right|
 +2\widetilde{C}\left(1+\left\|X^{\lambda_s}_L\right\|_\infty\right)\right)\left(C_1\widetilde{C}+\frac{1}{C_0}\right)+C_1\widetilde{C}\\
 &\qquad=:B.
\end{align*}
To finally get a bound for $\eta$ we need to have
\begin{equation}\label{conditionetabound}
 \underbrace{\widetilde{C}\left\|X^{\lambda_s}_H\right\|_\infty\left(2C_1\widetilde{C}+\frac{2}{C_0}+C_1\right)}_{=:A}<1\quad\text{and}\quad
 \widetilde{C}<\frac12.
\end{equation}
Let $\rho > 0$ denote a value so that  \eqref{conditionetabound} is satisfied for $\widetilde{C}=\rho$.

Now, for $\widetilde{C}< \rho$ we get
that 
\[
 \left\|\eta\right\|_\infty\leq\frac{1}{1-A}B=:C_\eta
\]
and thus the first assertion of the lemma. Combining this with \eqref{eq:Ebound} gives the energy bound and thus completes the proof
\[
 E(u(s),\eta(s))< 2\left|a_--a_+\right|+2\widetilde{C}\left(1+2C_\eta\left\|X^{\lambda_s}_H\right\|_\infty+\left\|X^{\lambda_s}_L\right\|_\infty\right)=:C_E.
\]
\end{proof}

\subsection{Invariance}\label{sec:invariance}
With the help of Theorem \ref{continuationcompactness} we are now able to define a quasi-isomorphism

\begin{equation*}
\widetilde{\Phi}: \text{CF}(\mathcal{A}^{\hat{\varphi}}_{\left(\alpha_0,\alpha_1\right)}) \, \to \, \text{CF}(\mathcal{A}^{\hat{\varphi}}_{\left(\alpha_0,\alpha_2\right)}). 
\end{equation*}

First, we remark that $\lambda_s=f(s,x,r)\alpha_0$ and $J_s$ is SFT-like for $\lambda_s$ apart from a small neighborhood of the hypersurface $\Sigma\times\left\{1\right\}$. Outside of this neighborhood we have $\left\|\lambda_s\right\|_{J_s}=\sqrt{r}$.
Recall that by the maximum principle, Floer trajectories remain in $\Sigma\times(0,e^{2\nu})$, thus we can estimate

 \begin{align*}
 \widetilde{C}&=\sup_{s\in[0,1]}\sup_{\Sigma\times(0,+\infty)}\left\|\partial_ s\lambda_s(x,r)\right\|_{J_s}
 =\sup_{s\in[0,1]}\sup_{\Sigma\times (0,e^{2\nu}]}\left\|\partial_s f \alpha_0\right\|_{J_s}\\
 & =\sup_{s\in[0,1]}\sup_{\Sigma\times (0,e^{2\nu}]}\left\|\frac{\partial_s f}{f} f\alpha_0\right\|_{J_s}=\sup_{s\in[0,1]}\sup_{\Sigma\times (0,e^{2\nu}]}\left\|\frac{\partial_s f}{f} \lambda_s\right\|_{J_s}\\
 &\leq \sqrt{e^{2\nu}}\left\|\frac{\partial_s f}{f}\right\|_\infty+\epsilon<K\\
 \end{align*}
 for a constant $K\in \R$. Here, the $\epsilon$-Error comes from the fact that $J_s$ is not SFT-like on a small neighborhood.


Now, use an adiabatic argument to subdivide the homotopy into a family of homotopies $\left\{(\lambda_s^j,J_s^j)\right\}_{j}$, each of which satisfies the smallness assumption on 
$\sup_{s\in[0,1]}\sup_{\Sigma\times(0,e^{2\nu}]}\left\|\partial_s\lambda^j_s\right\|_{J^j_s}$.
Note that since the norms of the vector fields $X^{\lambda_s}_H$ and $X^{\lambda_s}_L$ are bounded, this still holds true after a subdivision of the one form and the almost complex structure.
%
%
%
Therefore, we can apply Theorem \ref{continuationcompactness} to get the desired compactness properties for $\mathcal{M}^{a_+}_{a_-}\left(\mathcal{A}^{\hat{\varphi}}_{\lambda_s^j},J_s\right)$ and we can define maps 
\[
\widetilde{\Phi_j}: \text{CF}\left(\mathcal{A}^{\hat{\varphi}}_{\lambda^j}\right) \, \to \, \text{CF}\left(\mathcal{A}^{\hat{\varphi}}_{\lambda^{j+1}}\right) 
\]
by 
\[
 \widetilde\Phi_j(z) = \sum_{w \in \crit\left(\mathcal{A}^{\hat{\varphi}}_{\lambda^{j+1}}\right)} \# \mathcal{M}^0\left(z,w;\mathcal{A}^{\hat{\varphi}}_{\lambda_s^j},J_s\right)w
\]
for  $z \in \crit(\mathcal{A}^{\hat{\varphi}}_{\lambda^j})$. Here $\mathcal{M}^0\left(z,w;\mathcal{A}^{\hat{\varphi}}_{\lambda^j_s},J_s\right)$ denotes the zero-dimensional part of the space of Floer trajectories from $z$ to $w$.

Furthermore, the $\widetilde{\Phi_j}$ descend to define continuation homomorphisms 
\[
\Phi_j: \RFH((M,\lambda^j);\hat{\varphi}) \, \to \, \RFH((M,\lambda^{j+1});\hat{\varphi}).
\]

With the inverse  homotopy $\widetilde{\lambda}^j_s = \lambda^{j+1} + \beta(s)(\lambda^{j+1} - \lambda^{j})$ we analogously get maps
\[
\widetilde{\Psi}_j: \text{CF}\left(\mathcal{A}^{\hat{\varphi}}_{\lambda^{j+1}}\right) \, \to \, \text{CF}\left(\mathcal{A}^{\hat{\varphi}}_{\lambda^{j}}\right).
\]


A homotopy of homotopies argument shows that $\Phi_j$ is an isomorphism.
This holds for every $j$ with $0 \leq j \leq n$, and we get:
\begin{prop}\label{Iso} 
Let $(\Sigma,\xi)$ be a hypertight contact manifold and let $\alpha_1,\alpha_2\in\mathcal{C}(\xi)$. Let $\lambda_1\in\Omega_{\nu}(\alpha_0,\alpha_1)$ and $\lambda_2\in\Omega_{\nu}(\Sigma,\xi)$. Then the homology groups $\RFH((M,\lambda_i);\hat{\varphi})$ are well-defined for $i=1,2$ and we have an isomorphism
\[
\RFH((M,\lambda_1);\hat{\varphi}) \, \cong \, \RFH((M,\lambda_2);\hat{\varphi}).
\]
\end{prop}
 Choosing ${\alpha_2=\epsilon\alpha_0}$ shows that $\RFH\left((M,\lambda_1);\hat{\varphi}\right)\cong\RFH\left((M,r\alpha_0);\hat{\varphi}\right)$, which is the Rabinowitz Floer homology of the actual symplectisation $(\Sigma\times (0,+\infty),r\epsilon\alpha_0)$ and thus we have proved Theorem \ref{mainthm}, using that the Rabinowitz Floer homology is invariant under rescaling of a contact form. Thus it is justifiable to write $\RFH(\Sigma,\xi;\hat{\varphi})$ for a hypertight contact manifold $(\Sigma,\xi)$ .

\section{Applications}

In the previous section, we have elaborated, that for a hypertight contact manifold, the Rabinowitz Floer homology is independent of the supporting contact form. This allows us to deduce results for hypertight manifolds by using the knowledge of the Rabinowitz Floer homology groups for a convenient choice of supporting contact form, namely one that has no contractible Reeb orbits.

For that we first state some properties of Rabinowitz Floer homology that we need, for more details see \cite{AlbersFuchsMerry2013}. If we write $\RFH\left(\Sigma,\alpha_1;\hat{\varphi}\right)$, we mean the Rabinowitz Floer homology of the perturbed action functional on $(M,\lambda)$ where $\lambda \in \Omega_{\nu}(\alpha_0, \alpha_1)$.


\begin{enumerate}
	\item 
	For $\alpha_0$ a supporting contact form without any contractible Reeb orbits and $\hat{\varphi}\in\mathcal{P}\Cont_0\left(\Sigma,\xi\right)$,  the Rabinowitz Floer homology is canonically isomorphic to the singular homology
	\[
	 \RFH_*\left(\Sigma,\alpha_0;\hat{\varphi}\right)\cong\mathrm{H}_{*+n-1}\left(\Sigma;\Z_2\right).
	\]
	
	The Rabinowitz Floer homology is independent of $\hat{\varphi}\in\mathcal{P}\Cont_0\left(\Sigma,\xi\right)$ in the following sense: there are canonical isomorphisms
	\[
	 \Phi_{\hat{\varphi}}:\,\RFH\left(\Sigma,\alpha\right)\rightarrow\RFH\left(\Sigma,\alpha;\hat{\varphi}\right),
	\]
	where $\RFH\left(\Sigma,\alpha\right)=\RFH\left(\Sigma,\alpha;\operatorname{\id}\right)$.
	
	For $\hat{\varphi},\,\hat{\psi}\in\mathcal{P}\Cont_0\left(\Sigma,\xi\right)$ there is a canonical map
	\[
	 \Phi_{\hat{\varphi},\hat{\psi}}:\,\RFH\left(\Sigma,\alpha;\hat{\varphi}\right)\rightarrow\RFH(\Sigma,\alpha;\hat{\psi})
	\]
	such that $\Phi_{\hat{\psi}}=\Phi_{\hat{\varphi},\hat{\psi}}\circ\Phi_{\hat{\varphi}}$.
	
	In particular, $\RFH_n\left(\Sigma,\alpha;\hat{\varphi}\right)$ contains a class $[\Sigma_{\hat{\varphi}}]\neq 0$ which is defined by
	\[
	 \Phi_{\hat{\varphi},\hat{\psi}}\left([\Sigma_{\hat{\varphi}}]\right)=[\Sigma_{\hat{\psi}}]\quad\text{and}\quad [\Sigma_{\mathrm{id}}]=[\Sigma]\in\RFH_n\left(\Sigma,\alpha\right)\cong\mathrm{H}_{2n-1}\left(\Sigma;\Z_2\right).
	\]

	\item We denote by $\RFH^c\left(\Sigma,\alpha;\hat{\varphi}\right)$ the Rabinowitz Floer homology generated by the subcomplex of $(u,\eta)\in\crit\left(\mathcal{A}_{\left(\alpha_0,\alpha_1\right)}^{\hat{\varphi}}\right)$ with $\eta\leq c$. At a critical point $(u,\eta)$ the action value is exactly $\mathcal{A}^{\hat{\varphi}}_{(\alpha_0,\alpha_1)}(u,\eta)=\eta$, see Remark \ref{actionatcrit}, and thus the critical points with $\eta\leq c$ form indeed a subcomplex. Then the inclusion of critical points induces a map 
	\[
	 \iota^c_{\hat{\varphi}}:\,\RFH^c\left(\Sigma,\alpha;\hat{\varphi}\right)\rightarrow\RFH\left(\Sigma,\alpha;\hat{\varphi}\right).
	\] 
	In particular, for two paths $\hat{\varphi},\hat{\psi}$ there is a constant $K(\hat{\varphi},\hat{\psi})$ such that the map $\Phi_{\hat{\varphi},\hat{\psi}}$ from property (1) defines a map
	\[
	 \Phi_{\hat{\varphi},\hat{\psi}}:\RFH_*^c\big(\Sigma,\alpha_0,\hat{\varphi}\big)\rightarrow \RFH_*^{c+K(\hat{\varphi},\hat{\psi})}\left(\Sigma,\alpha_0,\hat{\psi}\right)
	\]
	for any $c\in\R$.
\end{enumerate}
 
\begin{rem}\phantomsection $ $
\begin{enumerate}
 \item[1.] For property (1) and (2) we use the fact that one can also define the $s$-dependent moduli spaces $\mathcal{M}^{a_+}_{a_-}\left(\mathcal{A}^{\hat{\varphi_s}}_{\lambda},J_s\right)$ for a family of $s$-dependent contactomorphisms and show that they are compact. For a family of compactly supported contactomorphisms, the estimates for the $s$-dependent functional are immediate.
 

 \item[2.] We can estimate the constant $K$ via the two contact Hamiltonians of $\hat{\varphi},\hat{\psi}$. This estimate will be the content of Proposition \ref{prop:Kestimate}.
\end{enumerate}
\end{rem}

\subsection{Translated points for hypertight contact manifolds}

Let $\varphi$ be a contactomorphism which is contact-isotopic to the identity and let $\rho:\,\Sigma\rightarrow(0,\infty)$ be such that $\varphi^*\alpha=\rho\alpha$. Recall that a point $x\in\Sigma$ is a \textbf{translated point of $\varphi$} if there exists $\eta\in\R$ such that
\[
 \varphi(x)=\varphi^\eta_R(x),\quad\rho(x)=1.
\]

\begin{rem}\label{rem:leafwiseintersection}
 We identify the hypersurface $\Sigma=\Sigma\times \{1\}$. A \textbf{leafwise intersection point} for $\phi_L^1$ is a point $(x,1)$ such that
 \[
  \phi_L^1(x,1)=\left(\varphi^{-\eta}_R(x),1\right),
 \]
 where $\phi_L$ is the Hamiltonian diffeomorphism associated to $L$ defined in \eqref{hamdiffeo}.
 
 A point $x\in\Sigma$ is a translated point of $\varphi_1$ if and only if $(x,1)\in M$ is a leafwise intersection point for $\phi_L^1$.
\end{rem}

To detect translated points with Rabinowitz Floer homology, we use the following analogue of \cite[Lemma 3.5]{AlbersFuchsMerry2013}.

\begin{lem}\label{crit=trans}
$(x,r,\eta)$ is a critical point of $\mathcal{A}_{(\alpha_0,\alpha_1)}^{\hat{\varphi}}$ if and only if $p:=x\left(\frac12\right)$ is a translated point for $\varphi$ with respect to $\alpha_1$ with time-shift $-\eta$ and we have
\[
\mathcal{A}_{(\alpha_0,\alpha_1)}^{\hat{\varphi}}(x,r,\eta)=\eta.
\]
\end{lem}

As an immediate consequence of Theorem \ref{mainthm} and property (1) about Rabinowitz Floer homology, we get
\begin{cor}\label{hypertight-invariance}
	$\RFH_*((M,\lambda_1);\hat{\varphi}) \, \overset{\ref{mainthm}}{\cong} \,  \RFH_*(\Sigma,\alpha_0;\hat{\varphi})   \, \overset{(1)}{\cong} \, \mathrm{H}_{*+n-1}(\Sigma).$
\end{cor}


\begin{rem}
	Note that $\RFH_*\left(\Sigma,\alpha_0;\operatorname{\id}\right)$ is the Rabinowitz Floer homology of the actual symplectisation of $\left(\Sigma,\alpha_0\right)$, $\left(\Sigma \times (0, +\infty),\omega\right)$ equipped with the symplectic form $\omega=d(r\alpha_0)$ as defined in \cite{AlbersFuchsMerry2013}.
\end{rem}

In order to prove the second statement of Theorem \ref{translated points for hypertight}, we need the following lemma.

\begin{prop}\label{genericalpha}
	Assume $\dim(\Sigma)\geq 3$.
	Let $\mathcal{R} \subset \Sigma$ be the union of the images of all contractible closed Reeb orbits on $(\Sigma,\alpha)$.
	Then for generic $\alpha$ the set  $\mathcal{C} = \{\varphi \in \Cont_0(\Sigma) \, | \, \forall x \in \crit(\mathcal{A}^{\hat{\varphi}}_{(\alpha_0, \alpha_1)}), \,x(\frac{1}{2}) \cap \mathcal{R} = \emptyset \}$  is a residual subset in $\Cont_0(\Sigma)$.
\end{prop}

The proof of the above Proposition goes as \cite[Theorem 3.3]{AlbersFrauenfelder2012b}.

\begin{proof}[Proof of Theorem \ref{translated points for hypertight}]
Let us first take any contact form $\alpha\in\mathcal{C}(\xi)$ and any contactomorphism $\varphi \in \Cont_0\left(\Sigma,\xi\right)$ and assume there is no translated point. This means that there are no critical points of the action functional which a posteriori implies that the action functional is trivially Morse. Thus, the Rabinowitz Floer homology is defined and equal to $\RFH(\Sigma, \alpha; \varphi) = 0$, which contradicts Corollary \ref{hypertight-invariance}. Hence, i) follows. 

If $\Sigma$ is one-dimensional, statement ii) is easy to prove and is left as an exercise to the reader. For nondegenerate $\varphi$ the perturbed Rabinowitz action functional is Morse-Bott and thus for generic $\varphi$, either the lower bound in ii) holds, or there is a translated point on a closed contractible Reeb orbit. Proposition \ref{genericalpha} above asserts that there is a generic set of contact forms $\alpha\in\mathcal{C}(\xi)$, such that the second alternative can be avoided by a generic $\varphi \in \Cont_0(\Sigma, \xi)$. Altogether ii) follows.  
\end{proof}



\begin{rem}
	If $\varphi$ has a translated point on a closed contractible Reeb orbit, then the loop which is given by first going along the contactomorphism $\varphi$ and then close up along the Reeb orbit is contractible. This is true because critical points of the perturbed Rabinowitz action functional are contractible.
\end{rem}

The proof of Proposition \ref{spectralcont} shows that for $\epsilon<$ smallest Reeb period of $\alpha$, then
 \newline $\RFH^{(-\epsilon,+\epsilon)}\left(\Sigma,\alpha;\hat{\varphi}\right)$ is well defined and equal to $\RFH\left(\Sigma,\alpha_0;\hat{\varphi}\right)$. Thus, if the oscillation norm of $\varphi$ is smaller than the smallest Reeb period of $\alpha$ then there is no translated point on a closed Reeb orbit which gives statement \rmnum{2}) in Remark \ref{rem:translated points}.

\begin{rem}
	We can also consider any Hamiltonian diffeomorphism $\phi^t:~M\rightarrow M$ which is supported in a neighborhood of $\Sigma$ and get the analogous result to Theorem \ref{translated points for hypertight} for leafwise intersections of $\phi^1$.
\end{rem}

\subsection{Existence of invariant Reeb orbits}

We can also use the nonvanishing of Rabinowitz Floer homology, to deduce the existence of a so called $\varphi$-invariant Reeb orbit for $\varphi\in\mathcal{S}\Cont_0\left(\Sigma,\xi\right)$ a strict contactomorphism. The details are below.

A contactomorphism $\varphi:\Sigma\rightarrow\Sigma$ is called \textbf{strict} if $\varphi^*\alpha=\alpha$. In fact, if $\varphi\in\mathcal{S}\Cont_0(\Sigma,\xi)$ is a strict contactomorphism then $\varphi$ commutes with the Reeb flow.

Moreover, a Reeb orbit $x:\R\rightarrow \Sigma$ is called \textbf{$\varphi$-invariant} if $\varphi(x(s))=x(s+\tau)$ for some $\tau\in\R\setminus\left\{0\right\}$ and all $s\in\R$.
In particular, if $\varphi$ is strict, then a translated point $x\in\Sigma$ gives rise to a $\varphi$-invariant orbit. Indeed, if $x$ is a translated point, then every point on the Reeb orbit $\left\{\left.\varphi_R^s(x)\right|\,s\in\R\right\}$ is also a translated point
\[
\varphi\left(\varphi_R^s(x)\right)=\varphi_R^s\left(\varphi(x)\right)=\varphi_R^s\left(\varphi_R^\eta(x)\right)=\varphi_R^{\eta+s}(x)=\varphi_R^\eta\left(\varphi_R^s(x)\right).
\]
Thus the Reeb orbit $\left\{\left.\varphi_R^s(x)\right|\,s\in\R\right\}$ is $\varphi$-invariant.

After the above consideration we have the following corollary of Theorem \ref{translated points for hypertight}.

\begin{cor}
 If $\varphi\in\mathcal{S}\Cont_0\left(\Sigma,\xi\right)$ is a strict contactomorphism then either there is a $\varphi$-invariant Reeb orbit (if $\tau\neq 0$) or there is a fixed point for $\varphi$ (if $\tau=0$).
\end{cor}



\subsection{Existence of non-contractible Reeb orbits}

In the following, we will explain how Theorem \ref{translated points for hypertight} implies the existence of non-contractible Reeb orbits provided there exists a positive loop of contactomorphisms of $\Sigma$.

We call a \textit{loop} of contactomorphisms a \textbf{positive loop} if the contact Hamiltonian associated to it is everywhere positive, see Definition \ref{def:contactham}.

\begin{rem}\label{rem:positive loop}
 The notion of a positive loop does not depend on the choice of contact form but only on the contact structure. This is true because the positivity of a loop $\hat{\varphi}$ is equivalent to the property that the vector field $\frac{d}{dt}\varphi_t$ should define the given coorientation of the contact structure $\xi$ and this does not depend on the choice of contact form.
\end{rem}

Most of this subsection is analogous to \cite[Section 3 and 4]{AlbersFuchsMerry2013}. The results in this section are extensions of the results \cite[Proposition 4.6]{AlbersFuchsMerry2013} in the sense that we use an arbitrary contact form.

First, we introduce the notion of spectral numbers. Recall the map from property (2) about Rabinowitz Floer homology
\[
 \iota^c_{\hat{\varphi}}:\,\RFH^c\left(\Sigma,\alpha;\hat{\varphi}\right)\rightarrow\RFH\left(\Sigma,\alpha;\hat{\varphi}\right).
\]

\begin{defn}
 Let $\hat{\varphi}\in\mathcal{P}\Cont_0\left(\Sigma,\xi\right)$ be nondegenerate. Then its \textbf{spectral number} $c\left(\hat{\varphi};\alpha\right)\in\R$ is defined as
 \[
  c\left(\hat{\varphi};\alpha\right):=\inf\left\{\left.c\in\R\right|~\left[\Sigma_{\hat{\varphi}}\right]\in\iota^c_{\hat{\varphi}}\left(\RFH_*^c\left(\Sigma,\alpha;\hat{\varphi}\right)\right)\right\}.
 \]
\end{defn}

\begin{rem}\label{rem:spectral number}
This number is not always $-\infty$. Take $\hat{\varphi}=\operatorname{\id}$ and $\alpha=\alpha_0$ the contact form which does not possess any contractible Reeb orbits. We know that $[\Sigma_{\operatorname{\id}}]=[\Sigma]\in\RFH_n\left(\Sigma,\alpha_0;\operatorname{\id}\right)\cong \mathrm{H}_{2n-1}\left(\Sigma;\Z_2\right)$ and because there are no contractible Reeb orbits this class cannot be represented by a sequence of Reeb orbits. In particular, it cannot be represented by a sequence of Reeb orbits with arbitrarily long negative period and thus $c(\operatorname{\id};\alpha_0)=0$.
\end{rem}
Actually, after the observation that the spectral number is zero for $(\operatorname{\id};\alpha_0)$ we can deduce that $c(\operatorname{\id};\alpha)=0$ for any supporting contact form. This will be the content of Corollary \ref{spectral0}.

Recall that by property (1) of Rabinowitz Floer homology for two paths of contactomorphisms, $\hat{\varphi},\hat{\psi}\in\mathcal{P}\Cont_0\left(\Sigma,\xi\right)$ we can define a map
\[
\Phi_{\hat{\varphi},\hat{\psi}}:~\RFH^c\left(\Sigma,\alpha;\hat{\varphi}\right)\rightarrow\RFH^{c+K\left(\hat{\varphi},\hat{\psi}\right)}\left(\Sigma,\alpha;\hat{\psi}\right).
\]

For $K(\hat{\varphi},\hat{\psi})$ we have the following estimate.

\begin{prop}\label{prop:Kestimate}\cite{AlbersFuchsMerry2013}
	Let $l_t$ and $k_t$ denote the contact Hamiltonians of $\hat{\varphi}$ and $\hat{\psi}$ and $C(\hat{\varphi};\alpha)$ and $C(\hat{\psi};\alpha)$ the values defined in Lemma \ref{lem:Cadmissible} equation \eqref{eq:Cadmissible}. Then,
	\begin{equation}\label{eq:Kestimate}
	K(\hat{\varphi},\hat{\psi})\leq e^{\max\left\{C(\hat{\varphi};\alpha),C(\hat{\psi};\alpha)\right\}}\int_0^1\max\left\{\max_{x\in\Sigma}\left(l_t(x)-k_t(x)\right),0\right\}\,dt.
	\end{equation}
\end{prop}

We can estimate the difference between spectral numbers of two nondegenerate paths via this number $K$.

\begin{prop}\label{properties-spectral}
	Let $\hat{\varphi},\,\hat{\psi}\in\mathcal{P}\Cont_0\left(\Sigma,\xi\right)$ be two nondegenerate paths of contactomorphisms and denote again by $l_t$ and $k_t$ their contact Hamiltonians. Then we can estimate
	\[
	c(\hat{\psi};\alpha)\leq c\left(\hat{\varphi};\alpha\right)+K(\hat{\varphi},\hat{\psi}).
	\]
	
	 Furthermore, we have
	 \[
	 l_t(x)\leq k_t(x),\quad\forall x\in\Sigma,\,0\leq t\leq 1\quad\Rightarrow\quad c(\hat{\varphi};\alpha)\geq c(\hat{\psi};\alpha).
	 \]
	\end{prop}
	
	The proof of Proposition \ref{properties-spectral} is analogous to the proof in \cite[Proposition 4.2]{AlbersFuchsMerry2013}.

\begin{rem}
	   It is possible to extend $c$ to all of $\mathcal{P}\Cont_0\left(\Sigma,\xi\right)$  via the limit of nondegenerate paths in a unique manner and such that all the previously mentioned properties are still satisfied. 
\end{rem}

We equip  $\mathcal{P}\Cont_0\left(\Sigma,\xi\right)$ with the $C^2$-topology. For fixed $\alpha\in\mathcal{C}(\xi)$ the continuity of the map 
\begin{align*}
c:\,\mathcal{P}\Cont_0\left(\Sigma,\xi\right)&\rightarrow\R\\
\hat{\varphi}&\mapsto c(\hat{\varphi};\alpha)
\end{align*}
was proved in \cite[Lemma 4.3]{AlbersFuchsMerry2013}. We prove the following extension.

\begin{prop}\label{spectralcont}
	Let $\alpha\in\mathcal{C}(\xi)$. For any path $\hat{\varphi}$ of contactomorphisms, the map
	\begin{align*}
		(\hat{\varphi};\alpha)&\mapsto c(\hat{\varphi};\alpha)
	\end{align*}
	is continuous, were $\mathcal{C}(\xi)$ is equipped with the natural $C^0$-topology.
\end{prop}

As a corollary of Proposition \ref{spectralcont} we have

\begin{cor}\label{spectral0}
	For $\alpha\in\mathcal{C}(\xi)$ for $\xi$ a hypertight contact structure we have
	\[
		c(\operatorname{\id},\alpha)=0.
	\]
\end{cor}

	Proposition \ref{properties-spectral} together with Corollary \ref{spectral0} shows that $c(\hat{\varphi};\alpha)$ is finite for every nondegenerate $\hat{\varphi}$ and $\alpha$.

The full proof of Proposition \ref{spectralcont} can be found in the Appendix \ref{fullproof}. Here, we only give a proof of Corollary \ref{spectral0} for which we need the following lemma, which gives a lower bound on the periods of all non-constant Reeb orbits appearing during a homotopy of supporting contact forms.




\begin{lem}\label{T}
Let $\alpha_1,\alpha_2\in\mathcal{C}(\xi)$ and let $\alpha_s , \,  s\in [0,1]$, be a homotopy of supporting contact forms between $\alpha_0$ and $\alpha_1$. For $s \in [0,1]$, denote by $\mathcal{P}_s \subset \R\setminus\{0\}$ the set of periods of closed orbits of the Reeb vector field $R_{\alpha_s}$. And set $T_s := \inf\{ |\eta| \, | \, \eta  \in \mathcal{P}_s\}$. Then $T := \inf\{T_s \, | \, s \in [0,1]\} > 0$.

\end{lem}

\begin{proof}
Let $V$ be a non-vanishing vector field on a compact manifold $M$ and let $p \in M$. One can choose $\tau_p > 0$ and a neighborhood $U_p$ of $p$, such that the flow $\phi_V^t$ of $V$ satisfies $\phi_V^{\tau_p}(U_p) \cap U_p = \emptyset$. 
Now let the vector field $V_s$ depend continuously on a parameter $s$. Fix $s_0$ and choose as before $\tau_p$ and $U_p$ for the vector field $V_{s_0}$. Then for a sufficiently small interval $I_p$ containing $s_0$ it still holds that $\phi_{V_s}^{\tau_p}(U_p) \cap U_p = \emptyset$ for all $s \in I_p$. Choose a finite covering of $M$ by the sets $U_p$. We have a corresponding finite intersection $I_{s_0} = \bigcap I_p \neq \emptyset$ and a corresponding infimum $\tau_{s_0} = \inf{\tau_p} > 0$.  Then the smallest period of a closed orbit of every ${V_s}, \, s\in I$, is bounded from below by $\tau_{s_0}$.

Applying this to the family of Reeb vector fields $R_{\alpha_s}$, one finds for every $s_0 \in [0,1]$ a small interval $I_{s_0}$ containing $s_0$, such that $T_{s_0}=\inf\{T_s \,|\, s \in I_{s_0}\} >0$. Since $[0,1] = \bigcup_{k=1}^n I_{s_k}$ for some $s_k \in [0,1]$, $n \in \mathbb{N}$,  we conclude $T = \inf_{1 \leq k \leq n}T_{s_k} > 0$.
\end{proof}


Now, we prove Corollary \ref{spectral0}.

\begin{proof}[Proof of {\color{blue} Corollary \ref{spectral0}}]
		
	Let $\alpha_0$ be as above, $\lambda_0=r\alpha_0$. Let $\lambda \in \Omega_\nu(\alpha_0, \alpha_1)$ for some suitable $\nu > 0$.  Let $\lambda_s$ be a homotopy between $\lambda_0$ and $\lambda$ of the type considered before.  
	As in Lemma \ref{T} denote by $T>0$ the smallest period of a closed Reeb orbit for the homotopy $\alpha_s = \lambda_s|_{\Sigma \times \{1\}}$.

	As explained in Section \ref{sec:continuationmaps}, for fixed $s_0, s_1 \in [0,1]$ we can choose a reparametrized homotopy $\widetilde{\lambda}_s$ between $\lambda_{s_0}$ and $\lambda_{s_1}$ such that 
	\[
	\widetilde{C} := \sup_{s\in[0,1]}\sup_{\Sigma \times (0, e^{2\nu}]}\|\partial_s\widetilde{\lambda}_s\|_{J_s} < 2|s_1-s_0|\sup_{s\in[0,1]}\sup_{\Sigma \times (0, e^{2\nu}]}\|\partial_s\lambda_s\|_{J_s}.
	\]
	
	We can apply an adiabatic argument for this homotopy to get uniform bounds on $\eta$,  as in Lemma \ref{energybound},  by taking a small enough partition of $[0,1]$.  An analogous argument applies here.  
	
	We first show the following. 
	
	\begin{claim}
		Let $C_0$ and $C_1$ as in Lemma \ref{sfundlemma}.
		There is a constant $c_T$ depending on the original homotopy $\lambda_s$, $C_0$, $C_1$ and $T$, such that if $\widetilde{C} < c_T$, then
		
		\begin{align}\label{claim1}
		\mathcal{M}_{0}^{a_+}(\mathcal{A}^{\operatorname{\id}}_{\widetilde{\lambda}_s}) = \emptyset, \quad \text{ if } \quad a_+ \geq T
		\end{align}
		and
		\begin{align}\label{claim2}
		\mathcal{M}_{a_-}^{0}(\mathcal{A}^{\operatorname{\id}}_{\widetilde{\lambda}_s}) = \emptyset, \quad \text{ if } \quad a_- \leq -T.
		\end{align}
	\end{claim}
	
	
	
	
	To prove the claim we assume $a_+ \geq a_-$ and consider a trajectory $(u(s),\eta(s))$ in $\mathcal{M}_{a_-}^{a_+}(\mathcal{A}^{\operatorname{\id}}_{\widetilde{\lambda}_s})$. 
	Assume $\widetilde{C} < 1$.
	Then (see the proof of Lemma \ref{energybound})
	
	\begin{align*}
	a_+ - a_-  &\leq -E(u(s),\eta(s)) +  \left|\int_{0}^{1}\int_{S^1}(u(s))^*(\partial_s{\lambda}_s)\,ds\right| \\
	&\leq -E(u(s),\eta(s)) + \widetilde{C}\left(E(u(s),\eta(s))+1+\left\|\eta\right\|_{\infty}\sup_{s\in[s_0,s_1]}\|X^{\lambda_s}_H\|_{J_s}\right)  \\
	&\leq \widetilde{C}\left(1 + \left\|\eta\right\|_{\infty}\sup_{s\in[0,1]}\|X^{\lambda_s}_H\|_{J_s}\right) \\
	&\leq \widetilde{C}\left(1 + \frac{1}{1-A}B\left\|X^{\lambda_s}_H\right\|_{\infty}\right).
	\end{align*}

	Note that $\sup_{s\in[s_0,s_1]}\left\|X^{\lambda_s}_H\right\|_{J_s}\leq\sup_{s\in[0,1]}\left\|X^{\lambda_s}_H\right\|_{J_s}=\left\|X^{\lambda_s}_H\right\|_\infty$ is finite, where by $\left\{X^{\lambda_s}_H\right\}_{s\in[s_0,s_1]}$ we mean a possibly reparametrised homotopy.
	
	Here, we can estimate $A$ from the proof of Lemma \ref{energybound}
	
	\begin{align*}
	A=  \widetilde{C}\sup_{s\in[s_0,s_1]}\left\|X_{H}\right\|_{J_s}\left(2C_1\widetilde{C}+\frac{2}{C_0}+C_1\right)
	\leq \widetilde{C}\underbrace{\left\|X^{\lambda_s}_H\right\|_{\infty}(3C_1 + \frac{2}{C_0})}_{=:K} < \frac{1}{2},
	\end{align*}
	\begin{flushright}
		
	\end{flushright}
	if we assume $\widetilde{C} < \min\{1, \frac{1}{2K}\}$.
	
	Also 
	\begin{align*}
	B &\overset{\phantom{\widetilde{C} < 1}}{=} C_1\left(1+\max\left\{|a_+|,|a_-|\right\}\right)
	+\left(2\left|a_--a_+\right|
	+2\widetilde{C}\right)\left(C_1\widetilde{C}+\frac{1}{C_0}\right)+C_1\widetilde{C} \\ 
	&\overset{\widetilde{C} < 1}{\leq}\underbrace{\left(C_1+2\left(C_1+\frac{1}{C_0}\right)+C_1\right)}_{=:K_0}+C_1\max\left\{|a_+|,|a_-|\right\}+\underbrace{2\left(C_1+\frac{1}{C_0}\right)}_{=:K_1}|a_--a_+|\\
	&\overset{\phantom{\widetilde{C} < 1}}{\leq} K_0 + C_1 \max\left\{|a_+|,|a_-|\right\} + K_1 \left|a_--a_+\right|.    
	\end{align*}
	
	Hence
	\begin{align*}
	a_+ - a_- &\leq \widetilde{C}(1+2K_0\left\|X^{\lambda_s}_H\right\|_\infty + 2C_1\left\|X^{\lambda_s}_H\right\|_\infty \max\left\{|a_+|,|a_-|\right\} + K_1\left\|X^{\lambda_s}_H\right\|_\infty \left|a_--a_+\right|)\\
	&=\widetilde{C}\underbrace{\left(1+2K_0\left\|X^{\lambda_s}_H\right\|_\infty\right)}_{=: \widetilde{K}_0}+\widetilde{C}\underbrace{\left\|X^{\lambda_s}_H\right\|_\infty\cdot 2C_1}_{=:\widetilde{K}_1}\max\left\{|a_+|,|a_-|\right\} \\ 
	&\qquad +\widetilde{C}\underbrace{\left\|X^{\lambda_s}_H\right\|_\infty K_1}_{=:\widetilde{K}_2}|a_--a_+|.
	\end{align*}

	Now, choose $c_T>0$ so that 
	\[
	c_T < \min\left\{1, \frac{1}{2K}, \frac{T}{2\widetilde{K}_0},\frac{1}{4\widetilde{K}_1}, \frac{1}{4\widetilde{K}_2}\right\}.
	\]
	If $\widetilde{C} < c_T$, then	
	
	\begin{align*}
	a_+ - a_- < \frac{1}{2}T + \frac{1}{4}\max(|a_+|,|a_-|) + \frac{1}{4}|a_- -a_+|,
	\end{align*}
	and thus
	
	\begin{align*}
	\frac{3}{4}(a_+ - a_-) < \frac{1}{2}T + \frac{1}{4}\max(|a_+|,|a_-|).
	\end{align*}
	From this, the desired upper bounds follow. Indeed, if $a_-=0$, then $a_+<T$, and if $a_+=0$ then $a_->-T$. This concludes the proof of the Claim.
	

	Recall that the original homotopy $\lambda_s$ goes from $r\alpha_0$ to $\lambda$.
	Assume now $s_0$ and $s_1$ have been chosen such that $\widetilde{C} < c_T$.
	Let $[X]$ be a nonzero class in $\RFH(\Sigma, \lambda_{s_0}; \operatorname{\id})$ with
	\[
	[X] \in \iota_{\operatorname{\id}}^0(\RFH^0(\Sigma, \lambda_{s_0}; \operatorname{\id})), \quad [X] \notin \iota_{\operatorname{\id}}^{-T}(\RFH^{-T}(\Sigma, \lambda_{s_0}; \operatorname{\id})).
	\]
	
	Consider the non-zero class $\Phi_{s_0,s_1}([X]) \in \RFH(\Sigma, \lambda_{s_1}; \operatorname{\id})$. Result \eqref{claim1} from the previous claim  shows that it can be represented by chains consisting of elements in $\crit(\mathcal{A}_{\lambda_{s_1}}^{\operatorname{\id}})$ with nonpositive period, so
	\[
	\Phi_{s_0,s_1}([X]) \in \iota_{\operatorname{\id}}^0 (\RFH^0(\Sigma, \lambda_{s_1}; \operatorname{\id})).
	\]
	We can also conclude that 
	\[
	\Phi_{s_0,s_1}([X]) \notin \iota_{\operatorname{\id}}^{-T}(\RFH^{-T}(\Sigma, \lambda_{s_1}; \operatorname{\id})).
	\]
	Namely assume the contrary. So $\Phi_{s_0,s_1}([X])$ can be represented by chains consisting of critial points all of whose periods are $<-T$. Equation \eqref{claim2} from the claim for the inverse homotopy shows 
	\[
	\Psi_{s_1,s_0}(\Phi_{s_0,s_1}([X])) \in \iota_{\operatorname{\id}}^{-T}(\RFH^{-T}(\Sigma, \alpha_{s_0}; \operatorname{\id})).
	\]
	But since $\Psi_{s_1,s_0} \circ \Phi_{s_0,s_1} = \operatorname{\id}$,  it follows that 
	\[
	[X]\in\iota_{\operatorname{\id}}^{-T}\RFH^{-T}\left(\Sigma,\alpha_{s_0};\operatorname{\id}\right)
	\]
	
	which contradicts the assumption about $[X]$.
	
	To conclude the proof, note that $[\Sigma] \in \RFH(\Sigma, \alpha_0; \operatorname{\id})$ satisfies exactly the assumption about $[X]$ above in fact $[\Sigma]\in\RFH^{(-\epsilon,+\epsilon)}\left(\Sigma,\alpha_0;\operatorname{\id}\right)$ because there are no contractible Reeb orbits for $\alpha_0$. Hence the same follows for $\Phi([\Sigma]) \in \RFH(\Sigma, \alpha, \operatorname{\id})$. We conclude that
	\[
c(\mathrm{\id};\alpha) = \inf\left\{\left.c\in\R\right|~\left[\Sigma_\alpha\right]\in\iota^c_{\id}\left(\RFH_*^c\left(\Sigma,\alpha;\id\right)\right)\right\} = 0.
\]	
\end{proof}

\begin{rem}\label{rem:spectralcont}
	For $\hat{\varphi}=\left\{t\mapsto \varphi_{R_{\alpha}}^{tT}\right\}$ or $\hat{\varphi}$ a loop,
	   analogous arguments as in the proof of Corollary \ref{spectral0} show that if we assume $\left\|\alpha_s\right\|_{J_s}$ to be sufficiently small 
	  then the corresponding spaces of Floer trajectories $\mathcal{M}^{-T+a}_{-T}(\mathcal{A}_{\lambda_s}^{\hat{\varphi}_s})$ and $\mathcal{M}^{-T}_{-T+a}(\mathcal{A}_{\lambda_s}^{\hat{\varphi}_s})$ are empty for $a\notin(-\epsilon,\epsilon)$ (see Property (a) below). 
	  The main point about this argument is, that there is a spectral gap around $0\in\spec\left(\hat{\varphi}_s;\alpha_s\right)$. 
	  
	
	Meanwhile for $\hat{\varphi}$ any path of contactomorphisms (not of the type mentioned above), we have
	\[
	\spec\left(\hat{\varphi};\alpha\right)=\left\{\left.\eta\,\right|\,\hat{\varphi}\text{ has a translated point with time shift } -\eta \text{ w.r.t. } \alpha\right\}.
	\]
	Let $\alpha_0$ be without contractible Reeb orbits and $\eta\in\spec\left(\hat{\varphi}_0;\alpha_0\right)$. Let $\lambda_s$ be a homotopy as usual between $r\alpha_0$ and $\lambda\in\Omega(\alpha_0,\alpha_1)$ where $\alpha_1\in\mathcal{C}(\xi)$ and look at flow lines in $\mathcal{M}_\eta^\tau\left(\mathcal{A}_{\lambda_s}^{\hat{\varphi}_s}\right)$.
	Define
	\[
	\delta_s:=\inf\left\{\left.\tau_s-\eta\,\right|\tau_s\in\spec\left(\hat{\varphi};\alpha_s\right),\tau_s\neq\eta\,\right\}.
	\]
	Then for $\hat{\varphi}=\operatorname{\id}$ (or $\hat{\varphi}$ a loop, or a path along the Reeb flow) we can bound $\delta_s$ away from $0$. For general $\hat{\phi}$, however, it may happen that $\inf_s\delta_s=0$ and thus there is no spectral gap around $0$.
	
\end{rem}

Further, one can show that the spectral number satisfies the following properties. 

\begin{enumerate}
   \item[(a)] For any $\hat{\varphi}\in\mathcal{P}\Cont_0\left(\Sigma,\xi\right)$ the spectral number is a critical value of $\mathcal{A}_{(\alpha_0,\alpha)}^{\hat{\varphi}}$, that is $c(\hat{\varphi};\alpha)\in\spec(\hat{\varphi})$. In particular, for the Reeb flow $t\mapsto\varphi_{R_\alpha}^{tT}$ we have that $c(t\mapsto\varphi_{R_\alpha}^{tT};\alpha)=-T$. 
   Indeed, let $\hat{\varphi}$ be the path $t\mapsto\varphi_{R_\alpha}^{tT}$. The critical points are closed Reeb orbits and then one moves along the Reeb orbit for time $-T$.
		Thus
		\[
		\spec\left(\varphi_{R_\alpha}^{tT};\alpha\right)=\left\{-T+\left\{\text{set of contractible Reeb periods of }\alpha\right\}\right\}
		\]
		compare this also to Remark \ref{rem:leafwiseintersection}. From this it is clear that $c(\varphi_{R_{\alpha_0}}^{tT};\alpha_0)=-T$.
		
		For $(\hat{\varphi}_{R_{\alpha_s}};\alpha_s)$ a family it is still true that
		\[
		\spec\left(\varphi_{R_{\alpha_s}}^{tT};\alpha_s\right)=\left\{-T\right\}+\left\{\text{set of contractible Reeb periods of }\alpha_s\right\}
		\]
	
	By Lemma \ref{T} there is a constant $\epsilon>0$ so that the family of Reeb vector fields $R_{\alpha_s}$ has no nonconstant Reeb orbit with period in $(-\epsilon,+\epsilon)$ and hence
	\[
		\spec\left(id;\alpha_s\right)\cap(-\epsilon,+\epsilon)=\left\{0\right\}.
	\]
	Hence, $\left\{-T\right\}$ is an isolated component of the set $\spec\left(\varphi_{R_{\alpha_s}}^{tT};\alpha_s\right)$. 
	By continuity of the map $(\hat{\varphi};\alpha)\mapsto c(\hat{\varphi};\alpha)$, as proved in Proposition \ref{spectralcont}, and the fact that $c(\varphi_{R_{\alpha_0}}^{tT};\alpha_0)=-T$, we have
	\begin{align*}
		c\left(\varphi_{R_{\alpha_s}}^{tT};\alpha_s\right)=-T,\quad\forall s.
	\end{align*}

   \item[(b)] The map $c$ descends to a well defined map
	      \[
	       c:\,\widetilde{\Cont}_0\left(\Sigma,\xi\right)\rightarrow\R,
	      \]
	      where $\widetilde{\Cont}_0\left(\Sigma,\xi\right)$ denotes the the universal cover of $\Cont_0\left(\Sigma,\xi\right)$.
\end{enumerate}

In particular, using Proposition \ref{properties-spectral} and property (a) one can show the following.
\begin{cor}\label{cor:spectral number}
 Let $\hat{\varphi}\in\mathcal{P}\Cont_0\left(\Sigma,\alpha\right)$ with contact Hamiltonian $l_t$. Then, if
 \begin{align*}
  l_t<0, ~\forall t\in[0,1]&\quad\Rightarrow \quad c(\hat{\varphi};\alpha)>0\quad\text{and if}\\
  l_t>0,~\forall t\in[0,1]&\quad\Rightarrow\quad c(\hat{\varphi};\alpha)<0.
 \end{align*}
\end{cor}

The following argument is from \cite[Proof of Corlollary 4.4]{AlbersFuchsMerry2013}. 

\begin{proof}[Proof of Corollary \ref{cor:spectral number}]
	This follows now from Proposition \ref{properties-spectral} and property (a). 
	
	For $l_t<0$ there is $\epsilon>0$ so that $l_t\leq-\epsilon<0$. The constant function $\epsilon$ generates the path $\left\{t\mapsto\varphi_{R_\alpha}^{t\epsilon}\right\}$ which by property (a) has spectral number $-\epsilon$.
	Proposition \ref{properties-spectral} implies that for $l_t\leq-\epsilon$,
	\[
	c(\hat{\varphi};\alpha)\geq c(t\mapsto\varphi_{R_\alpha}^{-t\epsilon};\alpha)=\epsilon>0.	
	\]
	Analogously, we have for $l_t\geq \epsilon>0$
	\[
		c(\hat{\varphi};\alpha)\leq c(t\mapsto \varphi_{R_\alpha}^{t\epsilon};\alpha)=-\epsilon<0.
	\]
	
\end{proof}


Now, let $\hat{\varphi}=\left\{\varphi_t\right\}_{t\in S^1}$ be a \textbf{loop} of contactomorphisms. Then one can show a more specific result about the spectral number.

\begin{prop}\label{prop:noncontractible-orbit}
 Let $\hat{\varphi}$ be a loop of contactomorphisms. Then, if $c(\hat{\varphi};\alpha)\neq 0$ there exists a Reeb orbit of period $-c(\hat{\varphi};\alpha)$ which belongs to the free homotopy class $-u_{\hat{\varphi}}$.
 
 Also, $c(\hat{\varphi};\alpha)=0$ if and only if $-u_{\hat{\varphi}}$ is the class of contractible loops.
\end{prop}

The proof goes analogously to the proof of \cite[Lemma 4.3]{AlbersFuchsMerry2013}.

\begin{proof}[Proof of Proposition \ref{prop:noncontractible-orbit}]
 In the situation where $\hat{\varphi}$ is in fact a loop, then a critical point $(x,r,\eta)$  of $\mathcal{A}^{\hat{\varphi}}$ is precisely the concatenation of a closed Reeb orbit of period $\eta$ with the loop $t\mapsto\varphi_t(x)$ for some $x\in\Sigma$.
 
 Assume that $c(\hat{\varphi};\alpha)\neq 0$. Then, by property (a), there is a Reeb orbit with period $\eta=c(\hat{\varphi};\alpha)$. Again the loop $(x,r,\eta)$ is contractible and since it consists of the concatenation of a Reeb orbit with the orbit $x\mapsto\varphi_t(x)$ this means that the Reeb orbit lies in the free homotopy class $-u_{\hat{\varphi}}$.

 Now, assume that $c(\hat{\varphi};\alpha)=0$. Then, by property (a) of the spectral number, there must be a Reeb orbit with $\eta=0$ and thus there is a critical point which is of the form $(x,r,0)$ where $x(t)=\varphi_t(x)$. But since $\mathcal{A}^{\hat{\varphi}}$ is defined on the set of contractible orbits, this implies that $x(t)$ is contractible and thus $u_{\hat{\varphi}}$ is trivial.

To prove the converse, note that the same argument as in the proof of Proposition \ref{spectralcont} shows that for a loop $\hat{\varphi}$, $c(\hat{\varphi},\alpha_0)= 0$ implies  $c(\hat{\varphi},\alpha) = 0$.  
So if $c(\hat{\varphi},\alpha) \neq 0$, we have that $c(\hat{\varphi},\alpha_0) \neq 0$ and hence there is a Reeb orbit with respect to $\alpha_0$ in $-u_{\hat{\varphi}}$. The assumptions on $\alpha_0$ then imply that $-u_{\hat{\varphi}}$ is a non-trivial class.    
\end{proof}

Note that by Corollary \ref{spectral0} and Corollary \ref{cor:spectral number} we have for each positive loop
 \[
  c(\hat{\varphi};\alpha)<c(\operatorname{\id};\alpha)=0,~\forall \alpha\in\mathcal{C}(\xi).
 \]
Therefore, in this setting a positive loop is never contractible. 


As a corollary of Proposition \ref{prop:noncontractible-orbit} we have the following result which proves Theorem \ref{thm:noncontractible orbits}.

\begin{cor}\label{noncontractible-orbit}
	Let $\left(\Sigma,\xi\right)$ be a hypertight contact manifold. If there exists a positive loop $\hat{\varphi}\in\Cont_0\left(\Sigma,\xi\right)$, then for any supporting contact form there exists a closed Reeb orbit in the non-trivial free homotopy class $-u_\varphi$, i.e. there always exists a \textbf{non-contractible} Reeb orbit.
\end{cor}

\appendix
\label{appendix}
	
	\section{Full Proof of Proposition \ref{spectralcont}.}\label{fullproof}
	
	Here we proof continuity in the second component.
	
	Let $\alpha=\alpha_1\in\mathcal{C}(\xi)$ be any supporting contact form. 
	Let $\{\alpha_s\}_{s\in[1,2]}$ be a homotopy of supporting contact forms starting at $\alpha_1$, $\{\lambda_s\}_{s\in[1,2]}$ with $\lambda_s \in \Omega_{\nu}(\alpha_0, \alpha_s)$ a corresponding homotopy of $1$-forms, $\{J_s\}$ a homotopy of SFT-like almost complex structures compatible with $d\lambda_s$. We prove that the map $s \mapsto c(\hat{\varphi}, \alpha_s)$ is continuous in $s=1$. 
	Assume w.l.o.g. (see Chapter \ref{sec:continuationmaps}) that the continuation maps  
	\[
	\widetilde{\Phi}: \text{CF}\left(\mathcal{A}_{\lambda_1}^{\hat{\varphi}}\right) \, \to \, \text{CF}\left(\mathcal{A}_{\lambda_2}^{\hat{\varphi}}\right)
	\] 
	are well-defined and induce isomorphisms in homology. Let $C_0$ and $C_1$ be the constants determined in Lemma \ref{sfundlemma}.
	
	Note that if we choose $\sigma>1$ sufficiently small we can find as in Section \ref{sec:continuationmaps} a homotopy $\{\widetilde{\lambda}_s\}_{s\in[1,2]}$ between $\lambda_1$ and $\lambda_{\sigma}$ such that $\widetilde{C}(\sigma) = \sup_{s\in[1,2]}\sup_{\Sigma \times (0, e^{2\nu}]} \|\widetilde{\lambda}_s\|_{J_s}$ is arbitrarily small. We denote by $\Phi_{1,\sigma}$ the corresponding isomorphism in homology. 
	
	We prove the following
	\begin{claim}
		For all $\epsilon>0$ and all $a \in \R$ there is $\delta=\delta(a,\epsilon) > 0$ such that if $\{\widetilde{\lambda}_s\}$ satisfies $\widetilde{C}(\sigma)  < \delta$ we have that
		
		\begin{align}\label{Claim1}
		\mathcal{M}_{a}^{a_+}(\mathcal{A}_{\widetilde{\lambda}_s}^{\hat{\varphi}}) = \emptyset, \quad \text{ for all } \, a_+ \, \text{ with } \quad a_+-a \geq \epsilon
		\end{align}
		and
		\begin{align}\label{Claim2}
		\mathcal{M}_{a_-}^{a}(\mathcal{A}_{\widetilde{\lambda}_s}^{\hat{\varphi}}) = \emptyset, \quad \text{ for all } \, a_- \, \text{ with } \quad a - a_- \geq \epsilon. 
		\end{align}
		
	\end{claim}

	We proof the claim. Fix $\sigma>1$, let $a_+,a_- \in \R$ with $a_+ \geq a_-$ and consider a trajectory $(u(s),\eta(s))$ in $\mathcal{M}_{a_-}^{a_+}(\mathcal{A}_{\widetilde{\lambda}_s}^{\hat{\varphi}})$. 
	Assume that we have $\widetilde{C}(\sigma) < 1$. Then (see the proof of Lemma \ref{energybound})
	
	\begin{align*}
	a_+ - a_-  &\leq -E(u(s),\eta(s)) +  \left|\int_{0}^{1}\int_{S^1}(u(s))^*(\partial_s{\lambda}_s)\,dt\,ds\right| \\
	&\leq -E(u(s),\eta(s)) + \widetilde{C}\left(E(u(s),\eta(s))+1+\left\|\eta\right\|_{\infty}\sup_{s\in[1,\sigma]}\|X^{\lambda_s}_H\|_{J_s}+\sup_{s\in[1,\sigma]}\|X^{\lambda_s}_L\|_{J_s}\right)  \\
	&\leq \widetilde{C}\left(1 + \left\|\eta\right\|_{\infty}\sup_{s\in[1,2]}\|X^{\lambda_s}_H\|_{J_s}+\sup_{s\in[1,2]}\|X^{\lambda_s}_L\|_{J_s}\right) \\
	&\leq \widetilde{C}\left(1 + \frac{1}{1-A}B\left\|X^{\lambda_s}_H\right\|_{\infty}+\|X^{\lambda_s}_L\|_{\infty}\right).
	\end{align*}

	Note that $\sup_{s\in[1,\sigma]}\left\|X^{\lambda_s}_H\right\|_{J_s}\leq\sup_{s\in[1,2]}\left\|X^{\lambda_s}_H\right\|_{J_s}=\left\|X^{\lambda_s}_H\right\|_\infty$ is finite and analogously for $X^{\lambda_s}_L$. 
	Here, we can estimate $A$ from the proof of Lemma \ref{energybound}
	
	\begin{align*}
	A=  \widetilde{C}\sup_{s\in[1,\sigma]}\left\|X_{H}\right\|_{J_s}\left(2C_1\widetilde{C}+\frac{2}{C_0}+C_1\right)
	\leq \widetilde{C}\underbrace{\left\|X^{\lambda_s}_H\right\|_{\infty}(3C_1 + \frac{2}{C_0})}_{=:K} < \frac{1}{2},
	\end{align*}
	\begin{flushright}
		
	\end{flushright}
	if we assume $\widetilde{C} < \min\{1, \frac{1}{2K}\}$.
	Also 
	\begin{align*}
	B &\overset{\phantom{\widetilde{C} < 1}}{=} C_1\left(1+\max\left\{|a_+|,|a_-|\right\}\right)
	+\left(2\left|a_--a_+\right|
	+2\widetilde{C}\left(1+\|X^{\lambda_s}_L\|_{\infty}\right)\right)\left(C_1\widetilde{C}+\frac{1}{C_0}\right)+C_1\widetilde{C} \\ 
	&\overset{\widetilde{C} < 1}{\leq}\underbrace{\left(C_1+2\left(1+\|X^{\lambda_s}_L\|_{\infty}\right)\left(C_1+\frac{1}{C_0}\right)+C_1\right)}_{=:K_0}+C_1\max\left\{|a_+|,a_-|\right\}+\underbrace{2\left(C_1+\frac{1}{C_0}\right)}_{=:K_1}|a_--a_+|\\
	&\overset{\phantom{\widetilde{C} < 1}}{\leq} K_0 + C_1 \max\left\{|a_+|,|a_-|\right\} + K_1 \left|a_--a_+\right|.    
	\end{align*}
	
	Hence
	\begin{align*}
	a_+ - a_- &\leq \widetilde{C}(1+2K_0\left\|X^{\lambda_s}_H\right\|_\infty + 2C_1\left\|X^{\lambda_s}_H\right\|_\infty \max\left\{|a_+|,|a_-|\right\} + K_1\left\|X^{\lambda_s}_H\right\|_\infty \left|a_--a_+\right|+\|X^{\lambda_s}_L\|_{\infty})\\
	&=\widetilde{C}\underbrace{\left(1+2K_0\left\|X^{\lambda_s}_H\right\|_\infty+\|X^{\lambda_s}_L\|_{\infty}\right)}_{=: \widetilde{K}_0}+\widetilde{C}\underbrace{\left\|X^{\lambda_s}_H\right\|_\infty\cdot 2C_1}_{=:\widetilde{K}_1}\max\left\{|a_+|,|a_-|\right\} \\ 
	&\qquad +\widetilde{C}\underbrace{\left\|X^{\lambda_s}_H\right\|_\infty K_1}_{=:\widetilde{K}_2}|a_--a_+|.
	\end{align*}
		
	Let $\widetilde{\delta} > 0$, we determine it later, and choose $\delta > 0$ with  
	\begin{align}\label{delta}
	\delta < \min\left\{1, \frac{1}{2K}, \frac{\epsilon}{2\widetilde{K}_0},\frac{1}{4\widetilde{K}_1}\widetilde{\delta}, \frac{1}{4\widetilde{K}_2}\right\}.
	\end{align}
	If we have $\widetilde{C}(\sigma) < \delta $, then	
	
	\begin{align*}
	a_+ - a_- < \frac{1}{2}\epsilon + \frac{1}{4}\widetilde{\delta}\max\{|a_+|,|a_-|\} + \frac{1}{4}|a_- -a_+|,
	\end{align*}
	and thus
	
	\begin{align}\label{a}
	\frac{3}{4}(a_+ - a_-) < \frac{1}{2}\epsilon + \frac{1}{4}\widetilde{\delta}\max\{|a_+|,|a_-|\}.
	\end{align}
	
	In the following we determine $\widetilde{\delta}$ and thereafter $\delta$ in \eqref{delta} such that the claim follows from \eqref{a} by contradiction.  
	It is sufficient to find $\delta_1(a, \epsilon)$ such that \eqref{Claim1} holds and $\delta_2(a, \epsilon)$ such that \eqref{Claim2} holds.
	We show the first assertion, the second being analogous.
	
	So, fix $a \in \R$. Let $a_+ \in \R$ with $a_+ - a \geq \epsilon$ and assume that there is a trajectory in $\mathcal{M}_{a}^{a_+}(\mathcal{A}^{\hat\varphi}_{\widetilde{\lambda}_s})$. So \eqref{a} holds with $a_-= a$.
	We distinguish the cases $a \geq 0$ and $a < 0$: 
	
	If $a \geq 0$, then $|a_+| > |a|$ and with $\widetilde{\delta}_1(a,\epsilon) = 1- \frac{a}{a+\epsilon}>0$ we have
	\[
	\widetilde{\delta}_1\max\{|a_+|,|a|\} \leq \left(1- \frac{a}{a_+}\right)a_+ = a_+ - a
	\]
	and therefore \eqref{a} with $\widetilde{\delta} = \widetilde{\delta}_1$ implies $a_+-a <\epsilon$.
	Hence, if we choose $\delta_1(a, \epsilon)$ accordingly in \eqref{delta}, we get a contradiction if $\widetilde{C}(\sigma) < \delta_1$. 
	
	If $a < 0$ choose $\widetilde{\delta}_1(a,\epsilon) = \min\{1,\frac{\epsilon}{|a|}\}$ and $\delta_1(a, \epsilon)$ accordingly in \eqref{delta} to get a contradiction if $\widetilde{C}(\sigma) < \delta_1$.
	Namely if $|a_+| \geq |a|$ we have that $|a_+|=a_+ <a_+ - a$. Since $\widetilde{\delta}_1\leq 1$, \eqref{a} implies $a_+-a < \epsilon$, a contradiction to our assumption.  
	If $|a_+| < |a|$ we get the same contradiction by \eqref{a}, since $\widetilde{\delta}_1\leq \frac{\epsilon}{|a|}$. 
	This proves the claim.
	

	The proof of the Proposition now proceeds as follows. 
	
	Let $c = c(\hat\varphi, \alpha_1)$. Let $\kappa_0 > 0$.
	Let $\epsilon > 0$ and consider $\hat\delta=\min_{\bar{c}\in\{c-\kappa_0, c+\kappa_0\}}\delta(\bar{c},\epsilon)$, where $\delta$ is given by the claim.
	Choose $\sigma_0>1$ such that there is for each $\sigma$ with $1<\sigma<\sigma_0$ a homotopy $\{\widetilde{\lambda}_s\}$ as described above with $\widetilde{C}(\sigma) < \hat{\delta}$. 
	
	Let $0 < \kappa < \kappa_0$ and
	let $[X]$ be a nonzero class in $\RFH(\Sigma, \lambda_1; \hat\varphi)$ with
	\[
	[X] \in \iota_{\hat\varphi}^{c+\kappa}(\RFH^{c+\kappa}(\Sigma, \lambda_{1}; \hat\varphi)), \quad [X] \notin \iota_{\hat\varphi}^{c-\kappa}(\RFH^{c-\kappa}(\Sigma, \lambda_{1}; \hat\varphi)).
	\]
	
	Consider the nonzero class $\Phi_{1,\sigma}([X]) \in \RFH(\Sigma, \lambda_{\sigma}; \hat\varphi)$. Result \eqref{Claim1} from the previous claim  shows that it can be represented by chain elements in $\crit(\mathcal{A}_{\lambda_{\sigma}}^{\hat\varphi})$ with period less then $c+\kappa + \epsilon$, so
	\[
	\Phi_{1,\sigma}([X]) \in \iota_{\hat\varphi}^{c+\kappa +\epsilon} (\RFH^{c+\kappa + \epsilon}(\Sigma, \lambda_{\sigma}; \hat\varphi)).
	\]
	We can also conclude that 
	\[
	\Phi_{1,\sigma}([X]) \notin \iota_{\hat\varphi}^{c-\kappa - \epsilon}(\RFH^{c-\kappa - \epsilon}(\Sigma, \lambda_{\sigma}; \hat\varphi)).
	\]
	Namely assume the contrary. So $\Phi_{1,\sigma}([X])$ can be represented by chain elements all having periods smaller or equal $c-\kappa -\epsilon$. Result \eqref{Claim2} from the claim for the inverse homotopy shows 
	\[
	\Psi_{\sigma,1}(\Phi_{1,\sigma}([X])) \in \iota_{\hat\varphi}^{c-\kappa}(\RFH^{c-\kappa}(\Sigma, \alpha_{1}; \hat\varphi)).
	\]
	But since $\Psi_{\sigma,1} \circ \Phi_{1,\sigma} = \operatorname{\id}$,  it follows that 
	\[
	[X]\in\iota_{\hat\varphi}^{c-\kappa}(\RFH^{c-\kappa}\left(\Sigma,\alpha_{1};\hat\varphi\right))
	\]
	
	which contradicts the assumption about $[X]$.
	
	To conclude the proof note that $[\Sigma_{(\hat\varphi,\alpha_1)}]$ satisfies precisely the assumption about $[X]$ above for all $\kappa$ with $\kappa_0 > \kappa > 0$. Since $\Phi_{1,\sigma}([\Sigma_{(\hat\varphi,\alpha_1})])=[\Sigma_{(\hat\varphi,\alpha_{\sigma})}]$ we conclude that
	for all $\sigma$ with $0<\sigma<\sigma_0$,
	\[
	|c(\hat{\varphi};\alpha_{\sigma}) - c(\hat{\varphi};\alpha_1)| < \epsilon.
	\]

\bibliographystyle{amsalpha}
\bibliography{references}

\end{document}